\documentclass[12pt]{amsart}
\usepackage{amsfonts, amssymb, amsmath, amsthm, eucal, latexsym, array, pictex}
\usepackage{fullpage}
\usepackage{verbatim}

\input{cyracc.def}

\font\tencyr=wncyr10 %scaled \magstephalf
\font\tencyi=wncyi10 %scaled \magstephalf
\font\tencysc=wncysc10 %scaled \magstephalf

\def\rus{\tencyr\cyracc}
\def\rusi{\tencyi\cyracc}
\def\rusc{\tencysc\cyracc}

\newtheorem{thm}{Theorem}%[section]
\newtheorem{lm}[thm]{Lemma} %%[section]
\newtheorem{cl}[thm]{Corollary}
\newtheorem{prop}[thm]{Proposition}

\theoremstyle{remark}
\newtheorem{ex}{Example}%[section]
\newtheorem{rmk}{Remark}%[section]
\newtheorem{qn}{Question}%%[section]

\theoremstyle{definition}
\newtheorem{df}{Definition}%[section]
\renewcommand{\iff}{if and only if }

\newcommand{\eus}{\EuScript}
\newcommand{\gt}{\mathfrak}

\newcommand{\ff}{\mathbb F}

\newcommand{\GL}{{\rm GL}}

\newcommand{\ind}{{\rm ind\,}}

\newcommand{\codim}{\mathrm{codim\,}}
\newcommand{\rk}{\mathrm{rk\,}}

\newcommand{\Lie}{\mathrm{Lie\,}}

\newcommand{\Hom}{\mathrm{Hom}}

\renewcommand{\Im}{{\rm Im\,}}

\newcommand{\ad}{\mathrm{ad\,}}
\newcommand{\Ad}{{\rm Ad}}

\newcommand{\g}{\mathfrak g}

\newcommand {\cN}{{\mathcal N}}

\newcommand {\cS}{{\mathcal S}}

\newcommand{\esi}{\varepsilon}

\newcommand{\mK}{{\mathbb K}}
\newcommand{\VV}{{\mathbb V}}

\renewcommand{\le}{\leqslant}
\renewcommand{\ge}{\geqslant}

\newfam\Bbbfam\newfam\eufam\newfam\eusfam%
\font\Bbbfont=msbm10 scaled 1200%
\font\bbbfont=msbm10 scaled 1000%
\font\Bbbsmallfont=msbm8%
\textfont\Bbbfam=\Bbbfont\scriptfont\Bbbfam=\Bbbsmallfont%%
\font\euzw=eufm10 scaled 1200%
\font\euac=eufm7 scaled 1200%
\font\euacc=eufm7 scaled 1000%
\textfont\eufam=\euzw\scriptfont\eufam=\euac%
\scriptscriptfont\eufam=\euacc%
\font\euszw=eusm10 scaled 1200%
\font\eusac=eusm7 scaled 1200%
\font\eusacc=eusm7 scaled 1000%
\textfont\eusfam=\euszw\scriptfont\eusfam=\eusac%
\scriptscriptfont\eusfam=\eusacc%

\def\varnothing{\hbox {\Bbbfont\char'077}}

\def\bbk{\hbox {\Bbbfont\char'174}}
\def\bkk{\hbox {\bbbfont\char'174}}

\newfam\eusfam%
\font\euszw=eusm10 scaled 1200%
\font\eusac=eusm7 scaled 1200%
\font\eusacc=eusm7 scaled 1000%
\textfont\eusfam=\euszw\scriptfont\eusfam=\eusac%
\scriptscriptfont\eusfam=\eusacc%
\begin{document}
\hfill {\scriptsize November 24, 2008} %%%% \today}
\vskip1ex

\title[Surprising properties of centralisers]
{Surprising properties of centralisers in classical Lie algebras}
\author{O.\,Yakimova}
\address{Mathematisches Institut, 
Universit\"at Erlangen-N\"urnberg,
Bismarckstrasse 1\,1/2,
91054 Erlangen  Germany}
\email{yakimova@mpim-bonn.mpg.de}
%\thanks{The author is supported by the Humboldt Foundation} 
%% and RFBR  Grant 05-01-00988}
\begin{abstract}
Let $\g$ be a classical Lie algebra, i.e., 
either $\gt{gl}_n$, $\gt{sp}_n$, or $\gt{so}_n$ and 
let $e$ be a nilpotent element of $\g$. 
We study various properties 
of centralisers $\g_e$.
The first four sections deal with rather elementary questions,
like  the centre of $\g_e$, commuting
varieties associated with $\g_e$, or 
centralisers of commuting pairs. The second half of the paper 
addresses problems related to different Poisson structures 
on $\g_e^*$ and symmetric invariants of $\g_e$. 
\vskip0.5ex

\noindent
{\sc R{\'e}sum{\'e}.} \ Soit $\g$ une alg{\`e}bre de Lie classique, i.e.,
$\gt{gl}_n$, $\gt{sp}_n$, ou $\gt{so}_n$, et
soit $e$ un {\'e}l{\'e}ment nilpotent de $\g$.
Nous {\'e}tudions dans cet article diverses propri{\'e}t{\'e}s 
du centralisateur $\g_e$ de $e$. 
Les quatre premi{\`e}res sections concernent
des probl{\`e}mes assez {\'e}l{\'e}mentaires portant sur le centre de $\g_e$,
la vari{\'e}t{\'e} commutante de $\g_e$,
ou encore des centralisateurs des paires commutantes.
La seconde partie aborde des questions li{\'e}es aux
diff{\'e}rentes structures de Poisson sur $\g_e^*$
et aux invariants sym{\'e}triques de $\g_e$. 
\end{abstract}
\maketitle
\maketitle

\tableofcontents

%%%%%%%%%%%%%%%%%%%%%%%%%%%%%%%%
\section*{Introduction}
%%%%%%%%%%%%%%%%%%%

Suppose that $G$ is a connected  
reductive 
%%classical 
%%%Lie %% algebra  
algebraic group defined over a
%%% n infinite  
field $\ff$ and $\gt g=\Lie G$.
%%%???? ${\rm char}\,\ff$ ????
For %an %nilpotent  element  
$x\in\gt g$ let  
%%%%The main object of our interest is the centraliser 
$\gt g_x$  denote the centraliser of $x$ in $\gt g$.
Due to the existence of the Jordan decomposition many questions about 
centralisers are readily reduced to nilpotent elements $e\in\gt g$.  
In this paper we restrict ourself to the case of classical $\gt g$
and study various  properties of centralisers.
The first four sections deal with rather elementary questions,
like {\it commuting varieties} associated with $\gt g_e$ or 
centralisers of commuting pairs. The second half of the paper 
addresses problems related to different Poisson structures 
on $\gt g_e^*$ and symmetric invariants of $\gt g_e$. 
It pursues  further an approach and some 
methods  of \cite{ppy}.

 In Section~1,
we introduce a basis of $\gt g_e$, which is used throughout 
the paper. 
Section~2 is devoted to the description of the 
centre of $\gt g_e$.
Let $\cN(\gt g)\subset\gt g$ be the nilpotent cone, i.e., the set 
of nilpotent elements. 
Let $\rk\gt g$ denote the rank of $\gt g$.
Answering a question of  Hotta and Kashiwara, Sekiguchi wrote 
a short note \cite{sek},
where he stated (without a proof)
that, for each  classical Lie algebra $\gt g$ and each 
$e\in\cN(\gt g)$, there exists 
$x\in\gt g_e$ such that 
the centraliser $\gt g_{(e,x)}=\gt g_e\cap\gt g_x$ 
is of dimension $\rk\gt g$. He addressed the same problem 
for the exceptional Lie algebras, but was not able to deal with the $E_8$ case
and overlooked one orbit in type $G_2$. 
%%%and missed one orbit in type $G_2$ for which there is no such $x$ in $\gt g_e$. 
Recently W. de Graaf \cite{graaf}  calculated (using computer)  
that in the exceptional Lie algebras there are 
only three nilpotent orbits $Ge$ such that 
$\dim\gt g_{(e,x)}>\rk\gt g$ for all $x\in\gt g_e$, one in $G_2$,
one in $F_4$, and one in $E_8$. 
In Section~3, we prove that,
for each $x$ in a classical Lie algebra 
$\gt g$, there is a {\it nilpotent} element $e\in\gt g_x$ such that 
$\dim\gt g_{(x,e)}=\rk\gt g$.

%\vskip0.5ex

In Section~\ref{commv}, we study {\it mixed} commuting varieties, 
$\gt C^*(\gt g_e)=\{(x,\alpha)\in\gt g_e{\times}\gt g_e^* \mid \alpha([x,\gt g_e])=0\}$, 
associated with centralisers. 
In contrast with the reductive case, these varieties can be
reducible. The simplest examples are 
provided by 
a minimal nilpotent element in 
$\gt{sl}_4$ (defined by  partition $(2,1,1)$)) and a nilpotent element 
$e\in\gt{sp}_6$  with Jordan blocks $(4,2)$. On the other hand, we prove that if 
$e\in\cN(\gt{gl}_n)$ has at most two Jordan blocks, then 
$\gt C^*(\gt g_e)$ is irreducible. 

The last four sections are devoted to the coadjoint representation 
of $\gt g_e$. In those sections we assume that  
the ground field $\ff$ is algebraically closed and 
of characteristic zero. For a linear action of a Lie algebra 
$\gt q$ on a vector space $V$, let $\gt q_v$ denote the 
stabiliser of $v\in V$ in $\gt q$. 
Recall that 
$\ind\gt q=\,\min_{\gamma\in\gt q^*} \dim\gt q_\gamma$.  
Set
$$
\gt q^*_{\rm sing}:=\{\gamma\in\gt q^*\mid \dim\gt q_\gamma>\ind\gt q\}.
$$ 
For a reductive Lie algebra $\gt g$ we have $\codim\gt g^*_{\rm sing}\ge 3$. 
In Section 5, the same is shown to be true for the centralisers in 
type $A$. In type $C$ there are elements such that 
$\codim(\gt g_e^*)_{\rm sing}=2$. In all other simple 
Lie algebras $\gt g$ 
the codimension of $(\gt g_e^*)_{\rm sing}$ may be $1$, see \cite[Section 3.9]{ppy}.

The dual space $\gt q^*$ of a Lie algebra carries a 
Poisson structure induced by the Lie-Poisson bracket on $\gt q$.
Having inequalities like $\codim\gt q^*_{\rm sing}\ge 2,3$
one can construct interesting (maximal) Poisson-commutative 
subalgebras in ${\mathcal S}(\gt q)$, see \cite{codim3}. 

By the Jacobson-Morozov theorem, $e$ can be included into an 
$\gt{sl}_2$-triple $(e,h,f)$ in $\g$.  Let us identify $\gt g$ and 
$\gt g^*$ by means of the  Killing form on $\g$.
Then $\gt g_e^*$ is isomorphic to a so called 
{\it Slodowy slice} $\mathbb S_e=e+\gt g_f\subset\gt g^*$
at $e$ to the (co)adjoint orbit $Ge$. 
The Slodowy slice ${\mathbb S}_e$, and hence
$\gt g_e^*$, carries another polynomial  Poisson 
structure, obtained from $\gt g^*$ via Weinstein reduction, 
see e.g. \cite{cush-rob} or \cite{gg}.
This second Poisson bracket is not linear in general and its linear part coincides 
with the Lie-Poisson bracket on $\gt g_e^*$. On the quantum level,
one can express 
the fact by saying that a {\it finite $W$-algebra}\  $W(\gt g,e)$
is a deformation of
the universal enveloping algebra ${\bf U}(\gt g_e)$. 
The centre of $W(\gt g,e)$ is a polynomial algebra in $\rk\gt g$ variables 
for all $\gt g$ and $e$. 
(It can be deduced from the 
analogous statement on the Poisson level, which is proved  
e.g. in \cite[Remark 2.1]{ppy}.)
In \cite{ppy}, the same is shown to be true for
the centre of ${\bf U}(\gt g_e)$, which is isomorphic to 
${\mathcal S}(\gt g_e)^{\gt g_e}$, if $\gt g$ is of type $A$ or $C$. 
In type $A$ another proof is given by Brown and Brundan \cite{bb}.
In Section~6, we compare construction of \cite{ppy} and \cite{bb} and 
conclude that they produce the same set of generating 
symmetric invariants.    

In Section~7, we prove that, in types $A$ and $C$, a generic fibre of the 
quotient morphism $\gt  g_e^*\to\gt g_e^*/\!\!/G_e$ consists of a single 
(closed) $G_e$-orbit. The most interesting fibre of this quotient  morphism 
is the one containing zero, the so called  
{\it null-cone} $\cN(e)$.  In type $A$ it is equidimensional 
by \cite[Section 5]{ppy}.  
Contrary to the expectations, see \cite[Conjecture 5.1]{ppy}, 
the null-cone is not reduced (as a scheme). 
A counterexample is provided by 
$e\in\cN(\gt{gl}_6)$ with partition $(4,2)$. 
This implies that the {\it tangent cone} at $e$ to 
$\cN(\gt{gl}_6)$   is not reduced either. For 
this nilpotent element there is an 
irreducible component of 
$\cN(e)$, which contains infinitely many closed  
$G_e$-orbits  and no regular elements. 

If $e\in\gt{gl}_n$ is defined by a rectangular 
partition $d^k$, then $\gt g_e$ is  
a truncated current algebra 
$\gt{gl}_k\otimes\ff[t]/(t^d)$ and it is also a 
so called  {\it Takiff} Lie algebra.  
As was noticed by Eisenbud and Frenkel \cite[Appendix]{mus},
a deep result of Must{\u{a}}{\c{t}}a \cite{mus} implies that 
 $\cN(e)$ is irreducible.  
Apart from that little is known about the number of 
irreducible components  of $\cN(e)$. 
We compute that $\cN(e)$ has $m{+}1$ components for the 
hook partition 
$(n, 1^m)$ with $n>1$, $m>0$, and $min(n-m,m){+}1$ components for 
the partition $(n,m)$ with $n\ge m$. 

Suppose that either $\gt g$ is an orthogonal Lie algebra and $e\in\gt g$
has only Jordan blocks of odd size or $\gt g$ is symplectic and $e$ has only  
Jordan blocks of even size. Then, as shown in Section~8,  all 
irreducible components of $\cN(e)$ are of dimension $\dim\gt g_e-\rk\gt g$. 
In type $A$ the same result 
is proved in \cite[Section~5]{ppy} for all nilpotent elements. 
  
In Sections~1--4, the ground field is supposed to be 
infinite and 
whenever dealing with orthogonal or symmetric Lie algebras we 
assume that ${\rm char}\,\mathbb F\ne 2$. 

\vskip1ex

\noindent {\bf Acknowledgements.} 
Parts of this work were carried out 
during my stay at the
Universit\"at zu K\"oln and 
Max Planck Institut f{\"u}r Mathematik (Bonn).  
I would like to thank both
institutions for hospitality and support. 
My gratitude goes to the anonymous referee for the 
careful reading of the paper and pointing out some 
inaccuracies. Thanks are also due to 
Alexander Elashvili and Alexander Premet for 
many inspiring discussions on the subject of nilpotent orbits.

\section{Basis of a centraliser}

The main object of this section is to introduce our notation. 
We construct a certain basis in $\gt  g_e$, which is used throughout 
the paper.  
Let $\VV$ be an $n$-dimensional vector space over
$\ff$ and let $e$ be a nilpotent element in
$\g=\gt{gl}(\VV)$. Let $k$ be the number of Jordan blocks of $e$
and $W\subseteq \VV$ a ($k$-dimensional) complement of $\Im e$ in
$\VV$. Let $d_i+1$ denote the size of the $i$-th Jordan block of
$e$. We always assume that the Jordan blocks are ordered such that
$d_1\ge d_2\ge\ldots\ge d_k$ so 
that $e$ is represented by the partition $(d_1{+}1,\ldots,d_k{+}1)$
of $n=\dim\VV$. 
Choose a basis $w_1, w_2, \ldots, w_k$ 
in $W$ such that the vectors $e^{j}{\cdot}w_i$ with $1\le
i\le k$, $0\le j\le d_i$ form a basis for $\VV$, and put
$\VV[i]:=\textrm{span}\{e^j{\cdot}w_i\,|\,\, j\ge 0\}$. Note that
$e^{d_i+1}{\cdot}w_i=0$ for all $i\le k$. 

If $\xi\in\g_e$, then $\xi(e^j{\cdot}w_i)=e^j{\cdot} \xi(w_i)$, hence
$\xi$ is completely determined by its values on $W$. Each vector
$\xi(w_i)$ can be written as
\begin{equation}
\xi(w_i)\,=\,\sum_{j,s} c_i^{j,s}e^s{\cdot}w_j,\ \qquad
c_i^{j,s}\in\ff.
\end{equation}
Thus, $\xi$ is completely determined by the coefficients
$c_i^{j,s}=\,c_i^{j,s}(\xi)$. This shows that $\g_e$ has a basis
$\{\xi_i^{j,s}\}$ such that
$$
\left\{
\begin{array}{l}
\xi_i^{j,s}(w_i)=e^s{\cdot}w_j, \\
\xi_i^{j,s}(w_t)=0 \enskip \mbox{for } t\ne i, \\
\end{array}\right.
\quad 1\le i,j\le k, \ \mbox{ and }\ \max\{d_j-d_i, 0\} \le s\le d_j \ .
$$
Note that $\xi\in\g_e$ preserves each $\VV[i]$ \iff $c_i^{j,s}(\xi)=0$
for $i\ne j$.

An example of $\xi_i^{j,1}$ with $i>j$ and $d_j=d_i{+}1$ is shown on 
Picture~1. On Picture~2, we indicate elements $\xi_i^{j,s}$ using 
Arnold's description of $\gt g_e$ for $e$ with three Jordan blocks. 
In that interpretation $e$ is given in a standard Jordan form and 
each $\xi_i^{j,s}$ as a matrix with entries $1$ on one of the (above) 
diagonal lines in one of the nine rectangles.

\begin{figure}[htb]
\setlength{\unitlength}{0.023in}
\begin{center}
\begin{picture}(90,85)(-9,-5)

\put(-9,6){$e\!\!:$}
\put(-2.7,5){\vector(0,1){10}}

\put(10,0){\line(1,0){20}}
\put(10,0){\line(0,1){70}}
\put(10,70){\line(1,0){20}}
\put(30,0){\line(0,1){70}}
\put(10,30){\line(1,0){20}}
\put(10,20){\line(1,0){20}}
\put(10,10){\line(1,0){20}}
\put(10,60){\line(1,0){20}}
\put(70,0){\line(0,1){60}}
\put(70,0){\line(1,0){20}}
\put(70,60){\line(1,0){20}}
\put(90,0){\line(0,1){60}}
\put(70,10){\line(1,0){20}}
\put(70,20){\line(1,0){20}}
\put(70,50){\line(1,0){20}}

\qbezier[20](20,32),(20,45),(20,58)
\qbezier[16](80,22),(80,35),(80,48)
\qbezier[15](50,23),(50,40),(50,57)

\put(70,05){\vector(-4,1){40}}                        
\put(70,15){\vector(-4,1){40}}  
\put(70,55){\vector(-4,1){40}}  

 \put(48,65){$\xi_i^{j,1}$}

{\small                                        
\put(11,63){$e^{d_j}{\cdot}w_j$}
\put(12.1,23){$e^2{\cdot}w_j$}
\put(13,13){$e{\cdot}w_{j}$}
\put(15.5,3){$w_{j}$}
\put(17.3,-5.5){$j$} 

\put(72,53){$e^{d_i}{\cdot}w_i$}
\put(74,13){$e{\cdot}w_{i}$}
\put(76.5,3){$w_{i}$}
\put(79,-5.5){$i$}  }
               
\put(31.4,2.6){$\ldots\ldots\ldots\ldots$}

\end{picture}
\end{center}
\caption{}\label{pikcha_A}
\end{figure}

\begin{figure}[htb]
\setlength{\unitlength}{0.023in}
\begin{center}
\begin{picture}(93,90)(-3,0)

\qbezier(0,0),(-3,45),(0,90)
\qbezier(90,0),(93,45),(90,90)

\put(0,20){\line(1,0){90}}
\put(0,50){\line(1,0){90}}
\put(70,0){\line(0,1){90}}
\put(40,0){\line(0,1){90}}

\put(1,89){\line(1,-1){38}}
\put(5,87){$e$}
\put(8.4,86.6){\line(1,-1){11.5}}
\put(20,72){$e$}
\put(23.4,71.6){\line(1,-1){12.2}}
\put(36,56){$e$}
\put(11,90){\line(1,-1){27.5}}
\put(41,49){\line(1,-1){28}}
\put(71,19){\line(1,-1){18}}
\put(23,79){$\xi_1^{1,s}$}
\qbezier(15.5,90),(19.25,86.25),(23,82.5)   %%\put(15.5,90){\line(1,-1){6}}
\qbezier(29,76.5),(34,71.5),(39,66.5) 
\qbezier[10](20,88),(27.5,88),(35,88) 
\qbezier(35,90)(37,88)(39,86)

%%% \xi_2^2
\put(45,46.5){$e$}
\put(48.4,46.6){\line(1,-1){18}}
\put(66,25.5){$e$}
\qbezier[4](51,45.5)(52.5,45.5)(54,45.5)
\put(55,43){$\xi_2^{2,s}$}
\qbezier(61.5,41.5)(64.5,38.5)(67.5,35.5)
\qbezier[2](65.5,45.5)(67,45.5)(68.5,45.5)

%%%%%% \xi_3^3
\put(75,16.5){$e$}
\put(78.4,16.6){\line(1,-1){11}}
\qbezier(79.5,19)(84.5,14)(89.5,9)
\qbezier[2](83.5,17)(84.5,17)(85.5,17)
\qbezier(85.5,19)(87.5,17)(89.5,15)

%%%%%%% \xi_2^1
\put(41,89){\line(1,-1){28}}
\put(53,83.5){$\xi_2^{1,s}$}
\qbezier[4](45.5,86)(48.75,86)(52,86)
\qbezier[2](63.76,86)(66.03,86)(68.3,86)
\qbezier(60,82)(64.5,77.5)(69,73)

%%%%%% \xi_3^1
\put(71,89){\line(1,-1){19}}
\put(78,83.5){$\xi_3^{1,s}$}
\qbezier(84.5,81.5)(86.75,79.25)(89,77)

%%%%%%%% \xi_3^2
\put(71,49){\line(1,-1){19}}
\put(78,43){$\xi_3^{2,s}$}
\qbezier(84.5,41.5)(86.75,39.25)(89,37)

%%%%%%%%%%%% \xi_1^2
\put(13,43){$\xi_1^{2,0}$}
\put(39,21){\line(-1,1){20}}
\put(23,43){$\xi_1^{2,1}$}
\put(39,31){\line(-1,1){10}}
\qbezier[4](33,45)(36,45)(39,45)

%%%%%%%%% \xi_1^3
\put(23,13){$\xi_1^{3,0}$}
\put(39,1){\line(-1,1){10}}
\qbezier[4](33,15)(36,15)(39,15)

%%%%%%%%% \xi_2^3
\put(53,13){$\xi_2^{3,0}$}
\put(69,1){\line(-1,1){10}}
\qbezier[4](63,15)(66,15)(69,15)

\end{picture}
\end{center}
\caption{}\label{pikcha_Arnold}
\end{figure}

A direct computation shows that the following commutator relation
holds in $\g_e$:
\begin{eqnarray}\label{commutator}
[\xi_i^{j,s},\xi]=\sum_{t,\ell}c_t^{i,\ell}(\xi)\xi_t^{j,\ell+s}-
\sum_{t,\ell}c_j^{t,\ell}(\xi)\xi_i^{t,\ell+s}\qquad\
(\forall\,\xi\in\g_e);
\end{eqnarray}
see \cite{fan} for more detail. 

Let $(\xi_i^{j,s})^*$ be a linear function on $\gt{g}_e$ such that 
$(\xi_i^{j,s})^*(\xi)=c_i^{j,s}(\xi)$. Then 
$\left<(\xi_i^{j,s})^*\right>$ form a basis of $\gt{g}_e^*$ dual to the basis 
$\left< \xi_i^{j,s}\right>$ of $\gt{g}_e$.  

Let $a\colon\, {\ff}^{^\times}\rightarrow
\GL(\VV)_e$ be the cocharacter such that $a(t){\cdot}w_i=t^iw_i$ for all
$i\le k$ and $t\in{\ff}^{^\times}$, and define a rational
linear action $\rho\colon\,{\ff}^{^\times}\rightarrow\GL(\g_e^*)$ by the formula
\begin{equation}          \label{ro}
\rho(t)\gamma=\,t\Ad^*(a(t)^{-1})\gamma \qquad\ \ \ \ \ \ \ \,
\big(\forall\,\gamma\in\g_e^*, \ \, \forall\,t\in \ff^{^\times}\big).
\end{equation}
Then $\rho(t)(\xi_i^{j,s})^*=t^{i-j+1}(\xi_i^{j,s})^*$ 
and for the adjoint action, denoted by the same 
letter, we have $\rho(t)\xi_i^{j,s}=t^{j-i-1}\xi_i^{j,s}$.

\vskip1ex

Let $(\ \,,\ )$ be a nondegenerate symmetric or
skew-symmetric bilinear form on $\VV$, i.e., 
$(v,w)=\esi(w,v)$, where $v,w\in\VV$ and $\esi=+1$ or $-1$.
Let $J$ be the matrix of  $(\ \,,\ )$ with respect to a basis $B$ of $\VV$. 
Let $X$ denote the matrix
of $x\in\gt{gl}(\VV)$ relative to $B$. The linear mapping $x\mapsto
\sigma(x)$ sending each $x\in\gt{gl}(\VV)$ to the linear
transformation $\sigma(x)$ whose matrix relative to $B$ equals
$-JX^t J^{-1}$ is an involutory automorphism of $\gt{gl}(\VV)$
independent of the choice of $B$. The elements of $\gt {gl}(\VV)$
preserving $(\ \, ,\ )$ are exactly the fixed points of $\sigma$. We
now set $\widetilde{\g}:=\gt{gl}(\VV)$ and let
$\widetilde{\g}=\,\widetilde{\g}_0\oplus\widetilde{\g}_1$ be the
symmetric decomposition of $\widetilde{\g}$ 
corresponding to the $\sigma$-eigenvalues $1$ and $-1$. 
%%% with respect to $\sigma$.  
The elements $x\in\widetilde{\g}_1$ have the property
that $( x{\cdot}v,w) =(v, x{\cdot}w)$ for all $v,w\in \VV$.

Set $\g:=\widetilde{\g}_0$ and let $e$ be a nilpotent element of
$\g$. Since $\sigma(e)=e$, the centraliser $\widetilde{\g}_e$ of
$e$ in $\widetilde{\g}$ is $\sigma$-stable and
$(\widetilde{\g}_e)_0=\,\widetilde{\g}_e^{\sigma}=\,\g_e$. This
yields the $\gt g_e$-invariant symmetric decomposition
$\widetilde{\gt g}_e=\gt g_e\oplus(\widetilde{\gt g}_e)_1$.

\begin{lm}   \label{restr}
In the above setting, suppose that $e\in\widetilde{\gt g}_0$ is a nilpotent element.
Then the cyclic vectors $\{w_i\}$ and thereby the spaces $\{\VV[i]\}$ can be chosen
such that there is an involution $i\mapsto i'$ on the set
$\{1,\dots, k\}$ satisfying the following conditions:  
\begin{itemize}
\item $d_i=d_{i'}$; 
\item $(\VV[i], \VV[j])=0$ if $i\ne j'$;
\item $i=i'$ if and only if $(-1)^{d_i}\esi=1$.
\end{itemize}
\end{lm}
\begin{proof} 
This is  a standard property of the nilpotent orbits in $\gt{sp}(\VV)$ and
$\gt{so}(\VV)$, see, for example, \cite[Sect.~5.1]{cm} or \cite[Sect.~1]{ja}.
\end{proof}

Let $\{w_i\}$ be a set of cyclic vectors %%%%associated with $e$ 
chosen according to Lemma~\ref{restr}. Consider the restriction of 
the $\gt g$-invariant form $(\ \,,\ )$  to $\VV[i]+\VV[i']$. 
Since $(w,e^s{\cdot}v)=(-1)^s(e^s{\cdot}w,v)$, 
a vector $e^{d_i}{\cdot}w_i$ is orthogonal to all 
vectors $e^s{\cdot}w_{i'}$ with $s>0$. Therefore 
$(w_{i'},e^{d_i}{\cdot}w_i)=(-1)^{d_i}(e^{d_i}{\cdot}w_{i'},w_i)\ne 0$.
There is a (unique up to a scalar)  vector $v\in\VV[i]$ such that 
$(v,e^s{\cdot}w_{i'})=0$ for all $s<d_i$. It is not contained  in 
$\Im e$, otherwise it would be orthogonal to $e^{d_i}{\cdot}w_{i'}$ too and hence 
to $\VV[i']$. Therefore there is no harm in replacing $w_i$ by $v$. 
Let us always choose the cyclic vectors $w_i$ in such a way 
that $(w_i,e^s{\cdot}w_{i'})=0$ for $s<d_i$ and 
normalise  them according to: %%%%  the following equalities:
\begin{equation}\label{normalise}
(w_i,e^{d_i}{\cdot}w_{i'})=\pm 1 \ \text{ and } \
(w_i,e^{d_i}{\cdot}w_{i'})>0 \ \text{ if } \ i\le i'. 
\end{equation}
Then $\gt g_e$ is generated (as a vector
space) by the vectors
$\xi_i^{j,d_j-s}+\varepsilon(i,j,s)\xi_{j'}^{i',d_i-s}$,
where $\varepsilon(i,j,s)=\pm 1$ depending on $i,j$ and $s$ in 
the following way
$$
(e^{d_j-s}{\cdot}w_j,e^{s}{\cdot}w_{j'})=-\esi(i,j,s)(w_i,e^{d_i}{\cdot}w_{i'}).
$$
Elements $\xi_i^{j,d_j-s}-\varepsilon(i,j,s)\xi_{j'}^{i',d_i-s}$ form a basis 
of $(\widetilde{\gt g}_e)_1$. In the following we always 
normalise $w_i$ as above and 
enumerate the Jordan blocks such that 
$i'\in\{i,i+1,i-1\}$ keeping inequalities $d_i\ge d_j$ for $i<j$.
The matrix of the bilinear form $(\ \,,\ )$ in this basis 
$\{e^s{\cdot}w_i\}$ is anti-diagonal with entries $\pm 1$. 

%%%%%%% SECTION %%%%%%%%%%%%%%%%%%%%%%%%%%% 
\section{The centre of a centraliser}
%%%%%%%%%%%%%%%%%%%%%%%%%%%%%%%%%%%%%%%%%%%
%%%%%%%%%%%%%%%%%%%%%%%%%%%%%%%%%%%%%%%%%%%

Let $\gt z$ be the centre of  $\gt g_e$. 
%% Consider $e$ as a matrix in $\gt{gl}(\VV)$. 
The powers of $e$ (as a matrix) are also elements of $\gt{gl}(V)$.
Set $\gt E:=\gt g\cap\langle e^0,e,e^2,\ldots,e^{d_1}\rangle_{\ff}$.
All higher powers of $e$ are zeros; the first element,  
$e^0$, is the identity matrix.  
Clearly, $\gt E\subset\gt z$. 
If $\gt g$ is either $\gt{sl}(\VV)$ or 
$\gt{sp}(\VV)$, then  
this inclusion is in fact the equality and 
in  orthogonal Lie algebras $\gt z$ can be larger.
For $\gt g$ classical, the centre of $\gt g_e$ was  
described  by Kurtzke \cite{kur} and that 
description is not quite correct. 

The following result is well-known. 
The proof is easy and illustrates the general scheme of argument
very well.

\begin{thm} If $\gt g=\gt{gl}(\VV)$ , then $\gt z=\gt E$.
\end{thm}
\begin{proof}
We have $e^s=\sum\limits_{i=1}^{k}\xi_i^{i,s}$ and $e^s\in\gt g$ 
for all $0\le s\le d_1$.
Suppose $\eta\in\gt z$.
Then $\eta$  commutes with 
the maximal torus $\gt t:=\langle\xi_i^{i,0}\rangle_{\ff}\subset\gt{gl}(\VV)_e$. 
We have 
$$
[\xi_i^{i,0},\xi_j^{t,s}]=\left\{
\begin{array}{rl}
-\xi_i^{t,s} &\mbox{ if } i=j, i\ne t; \\
\xi_j^{i,s} &\mbox{ if } i=t, i\ne j;\\
0 &\mbox{ otherwise.} \\
\end{array}\right.,
$$
Therefore $\eta\in\langle\xi_i^{i,s}\rangle_{\ff}$. 
Adding an element of 
$\gt E$ we may assume that $c_1^{1,s}(\eta)=0$ for all $s$.
If $\eta\not\in\gt E$, then there is some $c_i^{i,s}(\eta)$, which is not 
zero. Now take $\xi_1^{i,0}\in\gt g_e$ and compute that 
$$
[\eta,\xi_1^{i,0}]=c_i^{i,0}(\eta)\xi_1^{i,0}+c_i^{i,1}(\eta)\xi_1^{i,1}+\ldots+
c_i^{i,d_i}(\eta)\xi_1^{i,d_i}\ne 0.
$$
A contradiction! Thus $\gt z=\gt E$.
\end{proof}

\begin{cl} Suppose that $\gt g=\gt{sl}(\VV)$, then 
also  $\gt z=\gt E$. 
\end{cl}

\begin{thm}        \label{1}
If $\gt g=\gt{so}(\VV)$ and $e$ is given by a partition 
$(d_1+1,d_2+1,d_3+1,\ldots,d_k+1)$ with $k\ge 2$, where 
$d_2>d_3$ and both $d_1$ and $d_2$ 
are even, then $\gt z=\gt E\oplus\ff(\xi_1^{2,d_2}-\xi_2^{1,d_1})$.
For all other nilpotent elements of classical simple Lie algebras, 
we have $\gt z=\gt E$.
\end{thm}
\begin{proof} %%%% First consider the case $\gt g=\gt{sp}(V)$.
First we show that indeed in the special case indicated in the theorem we have 
an additional central element $x:=\xi_1^{2,d_2}-\xi_2^{1,d_1}$.
Note that $\xi_1^{2,d_2},\xi_2^{1,d_1}$ do not commute only with the elements 
$\xi_1^{1,0}$, $\xi_2^{2,0}$, $\xi_2^{1,d_1-d_2}$, and $\xi_1^{2,0}$.
Since $1'=1,2'=2$, the centraliser  
$\gt g_e$ contains no elements of the form $a\xi_1^{1,0}+b\xi_2^{2,0}$ and
we  have to check only that 
$[x,\xi_1^{2,0}+\esi(1,2,d_2)\xi_2^{1,d_1-d_2}]=0$.
Here $d_1$ and $d_2$ are even, therefore
$\esi(1,2,d_2)=-1$. We get 
$$
[x,\xi_1^{2,0}-\xi_2^{1,d_1-d_2}]=
   -\xi_1^{1,d_1}-\xi_2^{2,d_1}+\xi_1^{1,d_1}+\xi_2^{2,d_1}=0.
$$

Let us prove that $\gt z$ 
is not larger than stated in the theorem. 
The case $\gt g=\gt{gl}(\VV)$ (or $\gt{sl}(\VV)$) 
was treated above. Thus assume that $\gt g$ is either 
$\gt{sp}(\VV)$ or $\gt{so}(\VV)$. 
Then $\gt E$ is a vector space generated by all odd 
powers of $e$. 

Suppose that $\eta\in\gt z$. 
If $\eta$ preserve the cyclic spaces $\VV[i]$, then 
$\eta\in\gt E$. It can be shown exactly in the same way as 
in the $\gt{gl}(V)$ case.  
Note that whenever $i\ne i'$ 
there is an $\gt{sl}_2$-triple (subalgebra)
$\gt q_i=\langle\xi_i^{i,0}-\xi_{i'}^{i',0},
   \xi_i^{i',0},\xi_{i'}^{i,1}\rangle_{\ff}\subset \gt g_e$.
Equality $[\eta,\gt q_i]=0$ forces $c_i^{j,s}(\eta)=0$
whenever $i\ne i'$ (or $j\ne j'$) and $i\ne j$,  also $c_i^{i,s}(\eta)=c_{i'}^{i',s}(\eta)$
for $i\ne i'$. 

Assume that $\eta\not\in\gt E$.
Take the minimal $i$ such that there is a non-zero $c_i^{j,s}(\eta)$ with $j\ne i$.
(Necessary $i'=i$ and $j'=j$.)
Fix this  $i$ and take the minimal $j$, and then the minimal $s$, with this property.
Since $c_j^{i,d_j-d_i+s}(\eta)\ne 0$, we have also $j>i$ and therefore $j>1,1'$.
There is an element $\xi:=\xi_j^{1,d_1-s}+\esi(j,1,s)\xi_{1'}^{j,d_j-s}\in\gt g_e$.
Consider the commutator $[\xi,\eta]=\xi\eta-\eta\xi$. We are interested 
in the coefficient $a_i:=c_i^{1,d_1}([\xi,\eta])$. 
Since all coefficients $c_i^{1',r}(\eta)$ are zeros and $j\ne i$, we get
$$
a_i=c_i^{j,s}(\eta)-\delta_{i,1}\esi(j,1,s)c_j^{1,d_1-d_j+s}(\eta).
$$
In particular, if $i\ne 1$, then $\eta$ is not a central element. 
Therefore $i=1$.

In the symplectic case $d_1$ and $d_j$ are odd, hence
$d_j-s$ and $s$ have different parity and 
$\esi(j,1,s)\esi(1,j,d_j-s)=-1$. Thus $a_i=2c_i^{j,s}\ne 0$. 
We get a contradiction. 

The orthogonal case is more complicated. 
If $j>2$, then also $j>2'$ and 
$$
c_1^{2,d_2}([\xi_j^{2,d_2-s}+\esi(j,2,s)\xi_{2'}^{j,d_j-s},\eta])=c_1^{j,s}(\eta)\ne 0.
$$
Since $\eta\in\gt z$, we get $j=2$. If $d_3=d_2$, then $3'=3$ and there is 
a semisimple element $\xi_2^{3,0}-\xi_3^{2,0}\in\gt g_e$, 
which does not commute with $\eta$. 

It remains to consider only the special case $d_2>d_3$. 
There is no harm in replacing $\eta$ by 
$\eta-c_1^{2,d_2}(\eta)(\xi_1^{2,d_2}-\xi_2^{1,d_1})$. In other words, we may assume that 
$c_1^{2,d_2}(\eta)=0$ and thereby $s<d_2$. It is not difficult to see that 
$\eta$ does not commute either with 
$\xi_1^{2,1}+\xi_2^{1,d_1-d_2+1}$ or $\xi_1^{2,0}-\xi_2^{1,d_1-d_2}$, 
depending on the parity of $s$.
Thus if $\eta\not\in\gt E\oplus\ff(\xi_1^{2,d_2}-\xi_2^{1,d_1})$, then $\eta$ is not a
central element. This completes the proof.
\end{proof}

\begin{rmk}
In  \cite[Proposition 3.5]{kur}, Kurtzke 
overlooked nilpotent 
elements in $\gt{so}(\VV)$ such that 
 $\gt E$ is of codimension $1$ in $\gt z$ and $k>2$.
\end{rmk}
  
\section{Centralisers of commuting pairs}

By Vinberg's inequality, $\dim(\gt g_e)_\alpha\ge\rk\gt g$ for 
any $\alpha\in\gt g_e^*$.
A famous conjecture of Elashvili states that there is $\alpha\in\gt g_e^*$ 
such that $\dim(\gt g_e)_\alpha=\rk\gt g$. In the classical case, 
Elashvili's conjecture is proved in  \cite{fan}
and for the exceptional Lie algebras it is verified (with a computer aid) 
by W. de Graaf \cite{graaf}. In \cite{graaf}, de Graaf also showed 
that in the exceptional Lie algebras there are only three nilpotent orbits $Ge$ such that 
$\dim(\gt g_e)_x>\rk\gt g$ for all $x\in\gt g_e$. 
The result was predicted by Elashvili.

By a result of Richardson \cite{Rich-com}, 
the commuting variety
$\gt C(\gt g):=\{(x,y)\in\gt g{\times}\gt g \mid [x,y]=0\}$
 is irreducible for each reductive Lie algebra $\gt g$. 
It coincides with the closure of a $G$-saturation 
$G(\gt t,\gt t)$, where $\gt t\subset\gt g$ is a maximal torus. 
Hence $\dim(\gt g_e)_x\ge\rk\gt g$ for all $x\in\gt g_e$.
A general % common ???? 
belief is that in the classical Lie algebras there is always an 
element $x\in\gt g_e$ for which the equality holds.  
The statement even appeared in the literature without a proof, \cite{sek}. 
Here we prove a slightly stronger statement. %, suggested by Panyushev. 
Set $\gt g_{(e,x)}:=(\gt g_e)_x=\gt g_e\cap\gt g_x$.

\begin{thm}\label{pary} Suppose that $\gt g$ is a classical 
simple Lie algebra and $e\in\cN(\gt g)$. 
Then there is a \ {\tt nilpotent}\, element $x\in\gt g_e$ 
such that $\dim \gt g_{(e,x)}=\rk\gt g$.
\end{thm}
\begin{proof} 
({\bf 1}) \   If $\gt g=\gt{sl}_n$, then $e$ can be included into a so called 
{\it principal nilpotent} pair 
$(e,x)$, where  
$x:=\sum_{i=1}^{k-1} \xi_{i}^{i+1,0}$ and $\dim\gt g_{(e,x)}=n-1$, 
see \cite{gin}.

% \noindent
% {\it step free}\  -- \underline{the symplectic and orthogonal cases}\ \ \
({\bf 2}) \ Now assume that $\gt g\subset\gt{gl}(\VV)$ is either symplectic or 
orthogonal. % and $e\in\gt g$ is distinguished.
% In this case $i'=i$ for all $i$. 
The required element $x\in\gt g_e$ is defined as
$$ 
x:=\sum_{i=1}^{k-1} \xi_{i}^{i+1,0}+\esi(i,i{+}1,0)\xi_{(i+1)'}^{i',d_i-d_{i+1}}.
$$
Set $\widetilde{\gt g}:=\gt{gl}(\VV)$. 
Let $a\colon\, {\ff}^{^\times}\rightarrow
\GL(\VV)_e$ be the cocharacter such that $a(t){\cdot}w_i=t^iw_i$ for all
$i\le k$ and $t\in{\ff}^{^\times}$ (the same as in  (\ref{ro})).
Then   $\Ad(a(t)){\cdot}x=tx_+ + t^{-1}x_-$, where $x=x_+ + x_-$ and 
$x_+=\sum_{i=1}^{k-1} \xi_{i}^{i+1,0}$. 
As in part ({\bf 1}) of the proof, 
$e$ and $x_+$ form a principal nilpotent pair in $\widetilde{\gt g}$. 
Therefore $\dim\widetilde{\gt g}_{(e,x_+)}=n=\dim\VV$. 
Points $x_+ + t^2x_{-}$ with $t\in\ff^{^\times}$ form a dense 
(if $\ff$ is algebraically closed, then open) 
subset of the line $x_+ +\ff x_-$.  Hence,
using 
semi-continuity of dimension, one can show that also 
$\dim\widetilde{\gt g}_{(e,x)}\le n$. 
  
Consider a product of matrices $e^rx^l$ as an element of $\widetilde{\gt g}_e$.
Then  
$e^rx^l{\cdot}w_1=e^r{\cdot}w_l+v$, where 
$$
v\in \VV_1\oplus\ldots\oplus \VV_{l-1}\oplus
 \left< w_l,e^1{\cdot}w_l,\ldots, e^{r-1}{\cdot}w_l  \right>.  
$$
Hence $\dim\left<e^rx^l{\cdot}w_1\mid r,l\ge 0\right>=n$. 
Clearly each $e^rx^l$ is an element of $\widetilde{\gt g}_{(e,x)}$.
Therefore $\dim\widetilde{\gt g}_{(e,x)}\ge n$. 
Taking into account that  
$\dim\widetilde{\gt g}_{(e,x)}\le n$, we get the  equality 
$\dim\widetilde{\gt g}_{(e,x)}=n$.
The centraliser $\widetilde{\gt g}_{(e,x)}$ 
is the linear span of the vectors $e^rx^l$.   

Recall that there is a $\gt g$-invariant bilinear form on $\VV$ 
such that  
$(\xi{\cdot}v,w)=-(v,\xi{\cdot}w)$ and 
$(\eta{\cdot}v,w)=(v,\eta{\cdot}w)$
for all vectors $v,w\in\VV$, $\xi\in\gt g$, $\eta\in\widetilde{\gt g}_1$.
Hence $e^r x^l\in\gt g$ if  $r+l$ is odd and 
$e^rx^l\in\widetilde{\gt g}_1$ if $r+l$ is even. 
The centraliser of the pair $(e,x)$ in $\gt g$ is equal to 
the intersection $\widetilde{\gt g}_{(e,x)}\cap\gt g$, which has  
dimension $[(n{+}1)/2]=\rk\gt g$. 
%%%Hence $\dim\gt  g_{(e,x)}=\rk\gt g$. 
\end{proof}

\begin{rmk} Suppose that $y=y_s+y_n$ is the Jordan decomposition of $y\in\gt g$ 
and $\gt g$ is classical. Then $\gt g_y=(\gt g_{y_s})_{y_n}$ and 
$\gt g_{y_s}$ is a direct sum of the
centre and simple classical ideals. 
Therefore Theorem~\ref{pary} is valid for all 
(not necessary simple) classical Lie algebras $\gt g$ and all 
(not necessary nilpotent)
$y\in\gt g$.
\end{rmk}

\section{Commuting varieties}\label{commv}

With a non-reductive Lie algebra $\gt q$ one can associate two 
different commuting varieties. The usual one $\gt C(\gt q)$, consisting of commuting pairs 
$(\xi,\eta)\in\gt q{\times}\gt q$, appeared in the previous section. 
In this section we consider {\it mixed} commuting varieties 
$$
\gt C^*(\gt q):=\{(x,\alpha)\in\gt q{\times}\gt q^* \mid \alpha([x,\gt q])=0\},
$$ 
associated with centralisers. These varieties are closely related to some questions 
concerning 
rings of differential operators. 
Another way to define $\gt C^*(\gt q)$ is to say that  it is the zero fibre 
of the moment map $\gt q\times\gt q^*\to\gt q^*$. 
 
The usual commuting  variety 
$\gt C(\gt g_e)$ is not always irreducible,  see \cite{fan}.
Here we show that  
$\gt C^*(\gt g_e)$ 
can be reducible  as well, even if 
$\gt g$ is of type $A$. 
However,  let us start with examples outside of type $A$.
The first of them is 
related to the following property:

\begin{equation}\label{strange}
\gt q_{\rm  reg}\cap \left(\bigcup_{\alpha\in\gt q^*_{\rm reg}} \gt q_\alpha\right) 
  =\varnothing. 
\end{equation}

\noindent
Here $\gt q^*_{\rm reg}:=\gt q^*\setminus \gt q^*_{\rm sing}$
%%%% consists 
%%% of all regular linear functions on $\gt q$. 
%%%The set $\gt q_{\rm reg}$ is defined in the same way as 
%%%  $\gt q^*_{\rm reg}$, i.e., 
and 
$\xi\in\gt q_{\rm reg}$ if and only if %%% said to be {\it regular} if
the stabiliser $\gt q_\xi$ has the minimal possible dimension.  

\begin{prop} Suppose that  $\gt q$ satisfies~(\ref{strange}).
Then $\gt C^*(\gt q)$ is reducible. 
\end{prop}
\begin{proof} Clearly $U_1:=\gt C^*(\gt q)\cap(\gt q_{\rm reg}{\times}\gt q^*)$ 
and $U_2:=\gt C^*(\gt q)\cap(\gt q{\times}\gt q^*_{\rm reg})$ are open subsets of 
$\gt C^*(\gt q)$ and according to (\ref{strange}), $U_1\cap U_2=\varnothing$.
\end{proof}

\begin{ex}\label{sp42-com} 
Let $e\in\cN(\gt{sp}_6)$ be  defined by 
the partition $(4,2)$. Then $\gt g_e$ has a basis 
$$
\xi_1^{1,1}, \xi_2^{2,1}, \xi_1^{1,3}, \xi=\xi_{1}^{2,0}+\xi_2^{1,2},
\eta=\xi_1^{2,1}-\xi_{1}^{2,3}  
$$
with the only non-trivial commutators being
$[\xi,\xi_1^{1,1}]=[\xi_2^{2,1},\xi]=\eta$ and  $[\eta,\xi]=2\xi_1^{1,3}$.
Suppose that $\alpha\in(\gt g_e^*)_{\rm reg}$ and 
$x\in(\gt g_e)_\alpha$. 
Since $\alpha$ is regular, it is non-zero on 
$[\gt g_e,\gt g_e]=\left<\xi_1^{1,3},\eta\right>$. 
On the other hand $\alpha([x,\gt g_e])=0$,  hence 
$\dim [x,\gt g_e]\le 1$ and $\dim(\gt g_e)_x\ge 4>\rk\gt g$.
Therefore $x$ is not regular and condition (\ref{strange}) holds for 
$\gt g_e$.
 \end{ex}

\begin{rmk}
The simplest example of a Lie algebra satisfying condition (\ref{strange}) 
is a Heisenberg algebra. The centralisers of subregular elements (given by partitions 
$(2n{-}2,2)$) in  $\gt{sp}_{2n}$ also satisfy~(\ref{strange}).   
\end{rmk} 

The second example is slightly different. 

\begin{prop} Suppose that for each $\alpha\in(\gt g_e^*)_{\rm reg}$
the stabiliser $(\gt g_e)_\alpha$ consists of nilpotent elements, but 
$\gt g_e$ itself contains semisimple elements. Then $\gt C^*(\gt g_e)$ 
is reducible.
\end{prop}
\begin{proof}
Clearly $U_1:=\{(\gt g_e)_\alpha{\times}\{\alpha\} \mid \alpha\in(\gt g_e^*)_{\rm reg}\}$
is an open subset of $\gt C^*(\gt g_e)$. On the other hand, there is an
open subset in $\gt g_e$ containing no nilpotent  elements. Its 
preimage $U_2\subset\gt C^*(\gt g_e)$ is again an open subset. By our assumptions 
$U_1\cap U_2=\varnothing$.  
\end{proof}

\noindent
There are such nilpotent elements in the orthogonal Lie algebra. 

\begin{ex}\label{comv-so7} Let $e\in\cN(\gt{so}_7)$ be  defined by the
partition $(3,2,2)$. Then $x:=\xi_2^{2,0}-\xi_3^{3,0}\in\gt g_e$  is a 
semisimple element, which is unique up to conjugation and multiplication by scalars.   
Suppose that $\alpha\in\gt g_e^*$ is such that 
$(\gt g_e)_\alpha$ does not consist of nilpotent elements. 
Since $(\gt  g_e)_\alpha$ is the Lie algebra of an algebraic group
$(G_e)_\alpha$, it contains a semisimple element, we may assume that $x$.
Then $\alpha$ is zero on $[x,\gt g_e]$.  Note that the centraliser of $x$ 
in $\gt g_e$ is three
dimensional. More precisely, it is generated by $x$, $\xi_1^{1,1}$ and 
$\eta:=\xi_2^{2,1}+\xi_3^{3,1}$. 
Since $x$ is semisimple, 
$\alpha=a_1((\xi_2^{2,0})^*-(\xi_3^{3,0})^*)+
  a_2(\xi_1^{1,1})^*+a_3((\xi_2^{2,1})^*+(\xi_3^{3,1})^*)$, 
where $a_1,a_2,a_3\in\ff$.  
It not difficult to see that $(\gt g_e)_\alpha$ contains elements 
$\xi_1^{2,1}-\xi_3^{1,2}$, 
$\xi_1^{3,1}+\xi_2^{1,2}$, $\xi_1^{1,1}$, and, by the assumption, $x$. Hence 
$\dim(\gt g_e)_\alpha\ge 4$ and $\alpha\in(\gt  g_e^*)_{\rm sing}$.
\end{ex}

\begin{rmk} %%% Slightly modifying argument of Example~\ref{comv-so7}, 
It is possible to  show that 
if $e\in\cN(\gt{so}(\VV))$  is given by a partition 
$(d_1+1,\ldots,d_k+1)$ 
with $d_1$ being even and all other $d_i$ odd, then 
$(\gt g_e)_\alpha$ consists of nilpotent elements for each 
$\alpha\in(\gt  g_e^*)_{\rm reg}$. Note that $\gt g_e$ contains semisimple  
elements, if  $k>1$. 
\end{rmk}

Let us say that a point $\gamma\in\gt g_e^*$ is {\it generic} and 
$(\gt g_e)_\gamma$ is a {\it generic stabiliser} if 
there is an open subset  
$U_0\subset \gt g_e^*$ such that $(\gt g_e)_\delta$ is 
conjugate to $(\gt g_e)_\gamma$ for each $\delta\in U_0$. 
Suppose that $\gt g=\gt{gl}(\VV)$. 
Consider a point $\alpha=\sum_{i=1}^{k} a_i(\xi_i^{i,d_i})^*\in\gt g_e^*$, 
where $a_i$ are pairwise distinct non-zero numbers.  
Then, as was proved in \cite{fan}, $\alpha$ is a 
generic point 
in $\gt g_e^*$ and $\gt h:=(\gt g_e)_\alpha=\left<\xi_i^{i,s}\right>_\ff$ is 
a generic stabiliser for the coadjoint action of $\gt g_e$. 
Set $\gt h^*:=\left<(\xi_i^{i,s})^*\right>_\ff\subset\gt g_e^*$. 
Then $\{\gamma\in\gt g_e^* \mid\ad^*(\gt h)\gamma=0\}=\gt h^*$
and $\gt C_0:=\overline{ G_e(\gt h{\times}\gt h^*)}$ is an 
irreducible component of $\gt C^*(\gt g_e)$.  
Likewise, if $e\in\gt{sp}(\VV)$, then  $\gt C_0\cap\gt C^*(\gt{sp}(\VV)_e)$ 
is an irreducible component of the mixed commuting variety 
associated with $\gt{sp}(\VV)_e$.

%%%Return to type $A$.

\begin{ex}\label{cv-min} Let $e$ be a minimal nilpotent element 
in $\gt g=\gt{sl}_{n+2}$
with $n>1$.
Then the mixed commuting variety $\gt C^*(\gt g_e)$ has at least 
two irreducible components. 
\end{ex}
\begin{proof}
Let us include $e$ into an $\gt{sl}_2$-triple
$\left<e,h,f\right>$ in $\gt g$. Then $h$ defines a $\mathbb Z$-grading 
of $\gt g$:
$$
\gt g(-2)\oplus\gt g(-1)\oplus\gt g(0)\oplus\gt g(1)\oplus\gt g(2),
$$
where $\gt g(-2)=\mathbb F f$, $\gt g(2)=\mathbb F e$, 
$\gt g(-1)\subset\gt g_f$, and $\gt g(1)\subset\gt g_e$. 
The centraliser  
$\gt g_e$ is a semiderect product of 
$\gt{gl}_n=\gt g(0)_e$ and a (normal) Heisenberg 
Lie algebra $\gt n=V\oplus\ff e$, where 
$V=\gt g(1)\cong \ff^n{\oplus}(\ff^n)^*$ as a 
$\gt{gl}_n$-module.
Making use of the Killing form, we identify 
$\gt g_e^*$ and $\gt g_f$. Let $\chi_f$ be the 
element of $\gt g_e^*$ corresponding to $f$. 
Fix the $h$-invariant decomposition 
$\gt g_e^*=\gt{gl}_n^*\oplus V^*\oplus (\ff e)^*$.

The theory of $\gt{sl}_2$-actions tells us that 
  $V^*=\ad^*(V)\chi_f$ %%% where $\alpha\in (\ff e)^*$ is a non-zero function;
and that the stabiliser of a point $\gamma+0+\chi_f$, 
with $\gamma\in\gt{gl}_n^*$, 
is equal to $(\gt{gl}_n)_\gamma\oplus\ff e$. 
Let $N\subset G_e$ be the unipotent radical. Then 
$\Lie N=\gt n$ and  $N(\gt{gl}_n^*+\ff\alpha)$ is an open subset of
$\gt g_e^*$. Taking its preimage in $\gt C^*(\gt g_e)$, we obtain  
that the $N$-saturation 
$$
Y:=N\left\{ (\gt{gl}_n)_\gamma\oplus\ff e)\times(\gamma+0+\ff^*\chi_f) 
 \mid \gamma\in\gt{gl}_n^*\right\}
$$
is an open subset of $\gt C^*(\gt g_e)$. 
It is irreducible, because the usual commuting variety 
associated with $\gt{gl}_n$ ($\cong\gt{gl}_n^*$) is irreducible by a result of 
Richardson \cite{Rich-com}.  Thus $\overline{Y}$ is an irreducible component of 
$\gt C^*(\gt g_e)$. A generic point $\alpha\in\gt g_e^*$ can be chosen as 
$\alpha=\gamma+\chi_f$, where $\gamma$ is a generic point in $\gt{gl}_n^*$. 
Therefore $\overline{Y}$ coincides with the irreducible component $\gt C_0$ 
related to generic stabiliser. 

%It is not difficult to compute that 
%$\dim Y=n^2+3n+2$. Note that $e$ is zero on the complement of $Y$ in 
%$\gt C^*(\gt g_e)$. 

Suppose that $((x,y,z)\times(\gamma,\beta,\delta))\in Y$.
Then there is unique $\xi\in V$ such that
$\beta=\ad^*(\xi)\delta$. Hence $y=[\xi,x]$ by the construction of $Y$. 

Take a pair $((x,y,z)\times(\gamma,\beta,0))\in \gt g_e{\times}\gt g_e^*$.
It belongs to $\gt C^*(\gt g_e)$ if and only if 
 $(\gamma+\beta)([x+y,\gt{gl}_n])=0$ and 
$\beta([x,V])=0$. Fix $\beta\in V^*$ and 
$x\in(\gt{gl}_n)_\beta$. 
Then the second condition is automatically satisfied and 
the first one can be rewritten as 
$\ad^*(x)\gamma+\ad^*(y)\beta=0$. 
Varying  $\gamma$ we can get any element 
of $(\gt{gl}_n/(\gt{gl}_n)_x)^*$ on the first place in this sum. Thus, if 
$\ad^*(y)\beta$ is zero on $(\gt{gl}_n)_x$, i.e., if
$\beta([y,(\gt{gl}_n)_x])=0$, then there is $\gamma\in \gt{gl}_n^*$ such that 
$((x,y,z)\times(\gamma,\beta,0))\in\gt C^*(\gt g_e)$. 

Suppose that  
$((x,y,z)\times(\gamma,\beta,0))\in\overline{Y}$. 
%%%%
%%%Again there is a unique 
%%%%$\xi\in V$ such that  $\beta=\ad^*(\xi)\alpha$. 
Then there are curves $\{\xi(t)\}\subset V$,  and 
$\{x(t)\}\subset\gt{gl}_n$ such that 
$\lim_{t\to 0} x(t)=x$, 
$\lim_{t\to 0} \ad^*(\xi(t))t\chi_f=\beta$, and 
$\lim_{t\to 0} [\xi(t),x(t)]=y$.
 Clearly this is possible only if either $\beta$ or $x$ or $y$ is zero. 

If $n>1$, then there are non-zero $x\in\gt{gl}_n$ and $\beta\in V^*$
such that $x\in(\gt{gl}_n)_\beta$. Since $\ad^*(\gt{gl}_n)\beta\ne V^*$, 
 there is also a non-zero $y\in V$ such that 
$((x,y,z)\times(\gamma,\beta,0))\in\gt C^*(\gt g_e)$. 
Therefore $\gt C^*(\gt g_e)$ is reducible. 
%%%%%%%
%%%
\begin{comment}

I suspect that the closure of 
$$   
Y_1:=\{(x,y,z)\times(\gamma,\beta,0) \mid  \beta\in V^*_{\rm reg},
 x\in(\gt{gl}_n)_\beta, \beta([y,(\gt{gl}_n)_x])=0, 
 \ad^*(x)\gamma+\ad^*(y)\beta=0 \}
$$
is the second irreducible component of $\gt C^*(\gt g_e)$. 
It has the same dimension as $Y$. 
\end{comment}
\end{proof}

\begin{rmk} It seems that the mixed commuting 
variety $\gt C^*(\gt g_e)$ considered in Example~\ref{cv-min}
has exactly two irreducible components. The first one 
is $\overline{Y}$ and the closure of 
$$   
\{(x,y,z)\times(\gamma,\beta,0) \mid  (\gt{gl}_n)_\beta\cong\gt{gl}_{n-1},
 x\in(\gt{gl}_n)_\beta, \beta([y,(\gt{gl}_n)_x])=0, 
 \ad^*(x)\gamma+\ad^*(y)\beta=0 \}
$$
is the second.
\end{rmk}

If $n=1$, i.e., the minimal nilpotent element has only two Jordan blocks, 
then the argument of Example~\ref{cv-min} does not work. This is not a coincidence.
As we will prove below, 
$\gt C^*(\gt g_e)$ is irreducible for all nilpotent elements with at most two 
Jordan blocks. Similar result for $\gt C(\gt g_e)$ was obtained 
by Neubauer and Sethuraman in \cite{NS}.

\begin{thm}\label{irred} 
Suppose that $e\in\cN(\gt{gl}(\VV))$ has at most two Jordan blocks. Then 
the mixed commuting variety $\gt C^*(\gt g_e)$ is irreducible.
\end{thm}
\begin{proof} For regular nilpotent elements the statement is clear. 
Therefore assume that $e$ is given by a partition $(m,n)$ with $m\ge n$.
Let $\gt z$ be the centre of $\gt g_e$ and 
${\rm Ann}([\gt g_e,\gt g_e])\subset\gt g_e^*$ the annihilator of 
the derived algebra $[\gt g_e, \gt g_e]$. 
Suppose that $\{(\xi,\alpha)\}\in\gt C^*(\gt g_e)$.
Then also 
$(\xi+\gt z)\times (\alpha+{\rm Ann}([\gt g_e,\gt g_e])\subset\gt C^*(\gt g_e)$.
The centre $\gt z$ is the linear span of vectors 
$\xi_1^{1,s}+\xi_2^{2,s}$ with $0\le s< m$.  
The derived algebra $[\gt g_e,\gt g_e]$ is spanned 
by vectors $\xi_i^{j,s}$ with $i\ne j$ and 
$(\xi_1^{1,s}-\xi_2^{2,s})$. Let us choose
complementary subspaces to $\gt z$ (in $\gt g_e$) and 
to ${\rm Ann}([\gt g_e,\gt g_e])$ (in $\gt g_e^*$) consisting 
of the elements $\xi$ and $\alpha$ of the following form: 
%%% with  
$$
\begin{array}{l}
\xi=\sum\limits_{i=0}^{n-1} a_{i+1}\xi_1^{2,i}+
 \sum\limits_{i=0}^{n-1} c_{i+1}\xi_2^{2,i}+
\sum\limits_{i=0}^{n-1} b_{i+1}\xi_2^{1,i+m-n} \enskip \text{ and } \\[1.5ex]
\alpha=\sum\limits_{i=0}^{n-1} x_{i+1}(\xi_1^{2,i})^*+
 \sum\limits_{i=0}^{n-1} z_{i+1}( (\xi_1^{1,m-n+i})^*-(\xi_2^{2,m-n+i})^* ) +
\sum\limits_{i=0}^{n-1} y_{i+1} (\xi_2^{1,i+m-n})^*, \\
\end{array}
$$
for some  $a_i,b_i,c_i,x_i,z_i,y_i\in\ff$.
We will prove irreducibility for the set of   ``commuting" pairs
$(\xi,\alpha)$. 
%%%Such $\xi$'s form a complement to $\gt z$ and $\alpha$'s\  \ --- \ to 
%%%%${\rm Ann}([\gt g_e,\gt g_e])$. 

Set $X:=(x_1,\ldots,x_n)^t$, $Y:=(y_1,\ldots,y_n)^t$, and 
$Z:=(z_1,\ldots,z_n)^t$. Consider $X$, $Y$, and $Z$ as vectors 
of an $n$-dimensional vector space $W$. 
Let $A,\,B$, and $C$ be the upper triangular $n{\times}n$ matrices with entries 
$a_i,\,b_i$, and $c_i$ on the $i$th diagonal line. So the first line of $A$
is $(a_1,a_2,\ldots,a_{n})$, the second 
$(0,a_1,a_2,\ldots,a_{n-1})$, and so on. 
Note that these matrices lie in the centraliser 
$\gt{gl}(W)_{\hat e}$ of a regular nilpotent element $\hat e$. Hence they  
commute with each other.   
%%The column vectors $Y$ and  $Z$ are defined similarly.
The mixed commuting variety 
$\gt C^*(\gt g_e)$ is defined by equations of three types  
$\alpha([\xi,\xi_2^{2,s}])=0$, $\alpha([\xi,\xi_1^{2,s}])=0$, and 
$\alpha([\xi,\xi_2^{1,s}])=0$. Take the first of them with $s=0$. Then 
we get the following 
$\sum_{i=1}^{n} b_iy_i  -\sum_{i=1}^{n} a_ix_i=0$.   
The vector $\xi_2^{2,1}$ will give us that
$\sum_{i=1}^{n-1} b_iy_{i+1}=\sum_{i=1}^{n-1} a_i x_{i+1}$. 
In matrix terms this can be expressed as $AX=BY$. 
Explicitly writing down equations of all three types one can 
deduce that $\gt C^*(\gt g_e)$ is defined by
the matrix equations 
\begin{equation}\label{2reg}
AX=BY, \ CX=BZ, \  CY=AZ.
\end{equation}
Thus our problem is reduced to a simple exercise in linear algebra.
The following lemma solves this exercise and thereby completes the proof. 
\end{proof}

\begin{lm}\label{troyki}
Suppose that $W$ is an $n$-dimensional vector space and 
$\hat e\in\gt{gl}(W)$ is a regular nilpotent element. 
 Let $P$ be the set of six-tuples 
$(A,B,C;X,Y,Z)$, where $A,B,C\in\gt{gl}(W)_{\hat e}$, 
$X,Y,Z\in W$, satisfying equations~(\ref{2reg}). 
Then $P$ is irreducible. 
\end{lm}
\begin{proof}
Suppose that $\hat e$ is written in the normal Jordan form. 
Keep notation of Theorem~\ref{irred}.
Let $U\subset P$ be an open subset, where  
$b_1\ne 0$ or, which is the same, $\rk B=n$. 
Then 
$U=\{(A,B,C;X,B^{-1}AX,B^{-1}CX)\mid b_1\ne 0\}$ 
is a $4n$-dimensional irreducible affine variety. 
On $U$ the third equation $CY=AZ$ reduces 
to  $CB^{-1}AX=AB^{-1}CX$
and is satisfied automatically because $CB^{-1}A=AB^{-1}C$.

Equations~(\ref{2reg}) are invariant under simultaneous cyclic permutation of 
$(A,B,C)$ and $(Z,Y,X)$. Therefore we may consider only those solutions, where  
$\rk B\ge \max(\rk A,\rk C)$. 
Note  that $\rk B=n-d$ (with $d>0$) if and only if 
$b_1=\ldots=b_d=0$ and $b_{d+1}\ne 0$.
Set 
$$
P_d:=\{(A,B,C;X,Y,Z)\in P \mid \rk B=n-d, \rk A\le n-d,\rk C\le n-d\}.
$$
Our goal is to show that $P_d\subset \overline{U}$ for each $0<d\le n$.

Let $A'$ be the $(n-d){\times}(n-d)$ right upper corner of $A$ and 
$X':=(x_{d+1},\ldots,x_n)^t$. Define $B',C',Y'$, and $Z'$ in the same way. Then 
$P_d$ is defined by: 
$$
\begin{array}{l}
 b_1=\ldots=b_d=a_1=\ldots=a_d=c_1=\ldots c_d=0, \ \ \ b_{d+1}\ne 0; \\
  Y'=(B')^{-1}A'X', \  \text{ and } \ Z'=(B')^{-1}C'X' . \\
 \end{array} 
$$
Clearly $P_d$ is an irreducible affine variety and it contains an irreducible open subset 
$(P_d)^{\circ}$ where $x_n y_n z_n\ne 0$. It suffices to prove that 
 $(P_d)^{\circ}\subset\overline{U}$. Therefore assume that $x_n y_n z_n\ne 0$.
 We would like to replace  $A$ by $A+{\mathcal E}_A$, where 
 ${\mathcal E}_A\in GL(W)_{\hat e}$ is non-degenerate and 
``small'',  and do the same with 
$B$ and $C$. 
Since $x_n y_n z_n\ne 0$, the vectors 
$X$, $Y$, and $Z$ lie in the single 
open orbit of $GL(W)_{\hat e}$. In particular 
$Y={\mathcal E}_A Y$ and $Z={\mathcal E}_C X$ for some 
${\mathcal E}_A,{\mathcal E}_C\in GL(W)_{\hat e}$. 
Let $E\in GL(W)$ be the  identity matrix.
Then 
$(A+\lambda{\mathcal E}_A,B+\lambda E, C+\lambda{\mathcal E}_C;X,Y,Z)\in U$ 
for all $\lambda\in{\mathbb F}^{^\times}$. 
Taking limit with $\lambda$ tending to zero, we conclude that 
%%It tends to zero if $\lambda$ tends to zero and 
%%$B+\lambda{\mathcal E}_B$ is non-degenerate for all 
%%$\lambda\ne 0$.  
%%
%%
%%It is possible if  
%%${\mathcal E}_AX={\mathcal E}_BY$ and ${\mathcal E}_CX={\mathcal E}_BZ$
%%for some elements  ${\mathcal E}_A,{\mathcal E}_B,{\mathcal E}_C$ of
%%$GL(\VV)_{\hat e}$. This is the case if and only if $X$, 
%%$Y$ and $Z$ lie in the same orbit of
%% $GL(\VV)_{\hat e}$. 
%Therefore 
$(P_d)^{\circ}\subset\overline{U}$ and $P$
is irreducible.
\end{proof}

\begin{qn} Is it true that 
in the case of two Jordan blocks the defining ideal 
of $\gt C^*(\gt g_e)$ is generated by Equations~(\ref{2reg})?
Here the singularities of $\gt C^*(\gt g_e)$ form a 
subset of codimension $3$ (defined by the equation $a_1=b_1=c_1=0$).  
Maybe this can help to solve the problem. 
\end{qn}

\begin{rmk} Let $x=x_s+x_n$  be the Jordan decomposition of 
$x\in\gt{gl}_n$. 
Then $(\gt{gl}_n)_x$ is a sum of 
centralisers  $(\gt{gl}_{n_i})_{e_i}$, where all $e_i$ are nilpotent.
Suppose that each $e_i$ has at most two Jordan block. 
In that case $x$ is said to be {\it two-regular}, see \cite{NS}.
The  mixed commuting variety associated with 
$\gt g_x$ is a product of mixed commuting 
varieties associated with $(\gt{gl}_{n_i})_{e_i}$. Hence it is irreducible.
\end{rmk}

\begin{comment}
\ldots {\it question about the ideal. is it generated by our matrix equations? 
This commuting variety is smooth in codimension $2$.\  } \
\ldots\ldots \ \ 

\ldots\ldots \ \ consider the problem in the simplest case
 $n=3$, where our equations are just:
$$ 
 bz=cx, by=ax, az=cy.
$$ 
Choose an order: $b>a>c>x>y>z$. Then the highest monomials are 
$bz,by$, and $az$. It seems that we already have a Groebner basis.
{\it should be checked, recall the definitions } \ \ \ldots \ 
If so, the highest components are squares free and the ideal
should be radical. In the general case nothing is clear. \ldots \ldots
\end{comment}

\section{Poisson structures on the dual space of a centraliser}
\label{str}

From now on, we assume that $\mathbb F$ is algebraically closed and 
of characteristic zero.

By the Jacobson-Morozov theorem, $e$ can be included into an 
$\gt{sl}_2$-triple $(e,h,f)$ in $\g$.  By means of the  Killing form on $\g$, 
we identify $\gt g$ and $\g^*$. 
Consider $e$ as an element of $\gt g^*$ and 
let ${\mathbb S}_e$ denote the {\it Slodowy slice} $e+\g_f$ at $e$
to the coadjoint orbit $Ge$. 
The Slodowy slice ${\mathbb S}_e$ 
is a transversal slice to coadjoint $G$-orbits (symplectic leaves) 
in a sense of \cite{alan} and 
therefore  carries a {\it transversal} Poisson 
structure obtained from $\gt g^*$ by the Weinstein reduction, 
see e.g. \cite{cush-rob} or \cite{gg}. 
This Poisson structure, which is in general non linear,  
turns out to be polynomial \cite{cush-rob}.
For each element $F\in{\mathcal S}(\gt g)^{\gt g}$ its restriction 
$F\vert_{{\mathbb S}_e}$ lies in the centre 
${\eus Z}\mathbb F[{\mathbb S}_e]$ of the Poisson algebra 
$\mathbb F[{\mathbb S}_e]$. Moreover 
${\eus Z}\mathbb F[{\mathbb S}_e]$ is a polynomial algebra in $\rk\gt g$ variables 
generated by the restrictions ${F_i}\vert_{{\mathbb S}_e}$ for each generating
system of invariants $\{F_1,\ldots, F_{\rk\gt g}\}\subset{\mathcal S}(\gt g)^{\gt g}$,
see e.g. \cite[Remark 2.1]{ppy}.

The $G$-equivariance of the Killing form implies that $\g_e=[e,\g]^\perp$. On the
other hand, $\g=[e,\g]\oplus\g_f$ by the $\gt{sl}_2$-theory. 
Thereby  ${\mathbb S}_e$ is naturally 
isomorphic to $\gt g_e^*$ and 
${\mathbb F}[{\mathbb S}_e]\cong {\mathbb F}[\g_f]\cong {\mathcal S}(\gt g_e)$.
Remarkably, the linear part of the transversal Poisson 
structure on ${\mathbb S}_e$
gives us the usual Lie-Poisson bracket on $\gt g_e^*$, see e.g. 
\cite{cush-rob}. This leads to a natural construction of 
symmetric $\gt g_e$-invariants.  

For a homogeneous  $F\in{\mathcal S}(\gt g)$, 
let $^{e\!}F$ be the component of minimal degree  of the 
restriction $F\vert_{\mathbb S_e}$. (The restriction is not necessary homogeneous.)
Identifying $\mathbb F[\mathbb S_e]$ and $\mathbb F[\gt g_e^*]$, we consider 
$^{e\!}F$ as an element of ${\mathcal S}(\gt g_e)$. 

\begin{lm}\cite[Proposition~0.1.]{ppy} 
Keep the above notation. Then $^{e\!}F\in{\mathcal S}(\gt g_e)^{\gt g_e}$
for each homogeneous $F\in{\mathcal S}(\gt g)^{\gt g}$.  
\end{lm}

In types $A$ and $C$ it is possible to choose generating sets 
$\{F_1,\ldots,F_{\rk\gt g}\}\subset{\mathcal S}(\gt g_e)^{\gt g_e}$ such that 
the $^{e\!}F_i$'s are algebraically independent and generate the whole algebra of 
symmetric $\gt g_e$-invariants, see 
\cite[Theorems~4.2 and 4.4]{ppy}. The success is partially due to the fact 
that  in those two cases 
$\codim(\gt g_e^*)_{\rm sing}\ge 2$. In all other simple Lie algebras 
there are nilpotent elements, for which the codimension is $1$. 
Here we show that in type $A$ the codimension of 
$(\gt g_e^*)_{\rm sing}$ in $\gt g_e^*$ is %% also 
greater than or equal to  $3$.

Suppose that  $\gt g=\gt{gl}(\VV)$.
Then there are certain points 
$\alpha:=\sum_{i=1}^{k} a_i (\xi_i^{i,d_i})^*$,
with $a_i\in\mathbb F^{^\times}$ being pairwise distinct,   and  
$\beta:=\sum_{i=1}^{k-1}(\xi_{i+1}^{i,d_i})^*$ in 
$\gt g_e^*$ such that 
$(\ff\alpha\oplus\ff\beta)\cap(\gt g_e^*)_{\rm sing}=\{0\}$,
see \cite[Section 3]{ppy}.  

To prove that the codimension of $(\gt g_e^*)_{\rm sing}$ is greater than $2$, 
we need to find the third, linear independent with 
$\alpha$ and $\beta$,  regular point. 
The following is a slight modification of 
\cite[Proposition 3.2]{ppy}. 

\begin{lm} \label{gamma}
Suppose that $\gt g$ is of type $A$. Take 
$\gamma:=\sum\limits_{i=1}^{k-1} (\xi_i^{i+1,d_i+1})^*$.
Then $\gamma\in(\g_e^*)_{\rm reg}$.
\end{lm}
\begin{proof}
From (\ref{commutator}) and the definition of $\gamma$ it follows
that $\gamma([\xi_i^{j,s},\xi])=c_{j-1}^{i,d_j-s}(\xi)
-c_j^{i+1,d_{i+1}-s}(\xi)$ for all $\xi\in\g_e$. 
Suppose that ${\rm ad}^*(\xi)\gamma=0$. Then $\gamma([\xi,\g_e])=0$ forcing
$c_{j-1}^{i,d_j-s}(\xi)=c_j^{i+1,d_{i+1}-s}(\xi)$ for all
$i,j\in\{1,\ldots,k\}$ and all $s$ such that 
$\max(0,d_j-d_i)\le s\le d_j$.

We claim that $c_j^{i,s}(\xi)=0$ for $i>j$. Suppose for a
contradiction that this is not the case and take 
the maximal $j$ for
which there are $i>j$ and $0\le t\le d_i$ such that
$c_j^{i,t}(\xi)\ne 0$. Recall that, according to our convention,
$d_i\le d_j$. Moreover, $d_i\le d_{j+1}$, since $i\ge j+1$.
Set $s:=d_{j+1}-t$. Then $d_{j+1}-d_i\le s\le d_{j+1}$ and
$c_{j}^{i,d_{j+1}-s}(\xi)=c_{j+1}^{i+1,d_{i+1}-s}(\xi)$. As $j+1>j$ and
$i+1<j+1$, the right hand side of the equality is zero, forcing 
$c_{j}^{i,d_{j+1}-s}(\xi)=c_j^{i,t}(\xi)$ to be zero.

Now take $\xi_{i-1}^{i,s}\in\g_e$ with $0\le s\le d_{i}$. 
Since $\gamma([\xi,\xi_{i-1}^{i,s}])=0$, we have
$c_{i}^{i,d_{i}-s}(\xi)=c_{i-1}^{i-1,d_i-s}(\xi)$. Therefore,
$c_i^{i,t}(\xi)=c_{i-1}^{i-1,t}(\xi)= c_1^{1,t}(\xi)$ for 
$0\le t\le d_{i}$. In the same way one can show that
$c_{i+\ell}^{i,t}(\xi)=c_{i+\ell-1}^{i-1,t}(\xi)=
c_{1+\ell}^{1,t}(\xi)$ for $d_{i}-d_{i+\ell}\le t\le d_{i}$. 
Hence
$\xi$ is determined by a pair $(\ell,t)$, where 
$0\le\ell<k$ and $d_1-d_{\ell+1}\le t\le d_1$, and a scalar 
$c_{1+\ell}^{1,t}(\xi)$. Thus 
$\dim(\g_e)_\gamma\le\dim\VV$ and
$\gamma\in(\gt g_e^*)_{\rm reg}$. 
\end{proof}

\begin{cl} The stabiliser $(\gt g_e)_\gamma$ 
has a basis $\eta_{i,s}$ with $1\le i\le k$ and 
$d_1-d_i\le s\le d_1$, where 
$\eta_{i,s}=\xi_i^{1,s}+\xi_{i+1}^{2,s}+\ldots+\xi_{k}^{k-i+1,s}$.
\end{cl}

\noindent
For a nilpotent element with three Jordan blocks,
points $\alpha$, $\beta$, and $\gamma$ are shown on 
Picture~\ref{pikcha_alpha}. Here $\beta$ and 
$\gamma$ are sums of two matrix elements 
with coefficients $1$, and $\alpha$ is the sum 
with coefficients $a_1,a_2,a_3$. All of them are considered 
as elements of $\gt g_f$. On the same picture we remind 
Arnold's description of a generic element of $\gt g_e$ 
(see also Picture~2 in Section~1).

\begin{figure}[htb]
\setlength{\unitlength}{0.020in}
\begin{center}
\begin{picture}(93,90)(-3,0)

\qbezier(0,0),(-3,45),(0,90)
\qbezier(90,0),(93,45),(90,90)

\put(0,20){\line(1,0){90}}
\put(0,50){\line(1,0){90}}
\put(70,0){\line(0,1){90}}
\put(40,0){\line(0,1){90}}

\put(1,89){\line(1,-1){38}}
\put(5,87){$e$}
\put(8.4,86.6){\line(1,-1){11.5}}
\put(20,72){$e$}
\put(23.4,71.6){\line(1,-1){12.2}}
\put(36,56){$e$}
\put(11,90){\line(1,-1){27.5}}
\put(41,49){\line(1,-1){28}}
\put(71,19){\line(1,-1){18}}
\qbezier(15.5,90),(27.25,78.25)(39,66.5) 
\qbezier[10](20,88),(27.5,88),(35,88) 
\qbezier(35,90)(37,88)(39,86)
\put(0.4,52){$a_1$}

%%% \xi_2^2
\put(45,46){$e$}
\put(48.4,46){\line(1,-1){17.5}}
\put(66,25.5){$e$}
\put(50,49){\line(1,-1){18}}
\put(54,49){\line(1,-1){15}}
\qbezier[6](58.5,45.5)(63.5,45.5)(68.5,45.5)
\put(41,22){$a_2$}

%%%%%% \xi_3^3
\put(75,16){$e$}
\put(78.4,16){\line(1,-1){10.5}}
\qbezier(79.5,19)(84.5,14)(89.5,9)
\qbezier[2](83.5,17)(84.5,17)(85.5,17)
\qbezier(85.5,19)(87.5,17)(89.5,15)
\put(71,1.5){$a_3$}

%%%%%%% \xi_2^1
\put(41,89){\line(1,-1){28}}
\put(45,89){\line(1,-1){24}}
\put(49,89){\line(1,-1){20}}
\put(53,89){\line(1,-1){16}}
\qbezier[6](56,88)(61,88)(66,88)
\qbezier(67,89)(68,88)(69,87)
\put(40.5,51){$\gamma$}

%%%%%% \xi_3^1
\put(71,89){\line(1,-1){19}}
\put(75,89){\line(1,-1){15}}
\put(79,89){\line(1,-1){11}}
\qbezier[6](82.25,87)(86.125,87)(90,87)

%%%%%%%% \xi_3^2
\put(71,49){\line(1,-1){19}}
\put(75,49){\line(1,-1){15}}
\put(79,49){\line(1,-1){11}}
\qbezier[6](82.25,47)(85.75,47)(89.25,47)
\put(70.5,21){$\gamma$}

%%%%%%%%%%%% \xi_1^2
\put(39,21){\line(-1,1){28}}
\put(39,25){\line(-1,1){24}}
\put(39,29){\line(-1,1){20}}
\qbezier[8](22,47)(29,47)(36,47)
\qbezier(35,49)(37,47)(39,45)
\put(0,21){$\beta$}

%%%%%%%%% \xi_1^3
\put(39,1){\line(-1,1){18}}
\put(39,5){\line(-1,1){14}}
\put(39,9){\line(-1,1){10}}
\qbezier[4](32,18)(35.5,18)(39,18)

%%%%%%%%% \xi_2^3
\put(69,1){\line(-1,1){18}}
\put(69,5){\line(-1,1){14}}
\put(69,9){\line(-1,1){10}}
\qbezier[4](62,18)(65.5,18)(69,18)
\put(41,1){$\beta$}

\end{picture}
\end{center}
\caption{}\label{pikcha_alpha}
\end{figure}

\begin{thm}\label{c3A} 
If $\gt g=\gt{gl}(\VV)$ with $\dim\VV\ge 3$,  
then $\codim(\gt g_e^*)_{\rm sing}\ge 3$.
\end{thm}
\begin{proof}
If $e$ is a regular element, then 
$(\gt g_e^*)_{\rm sing}=\{0\}$ and 
the codimension of this subset is equal to 
$\dim\gt g_e=\dim\VV$. Suppose that $e$ is not regular and 
let elements $\alpha=\sum_{i=1}^{k} a_i(\xi_i^{i,d_i})^*$, 
$\beta=\sum_{i=1}^{k-1}(\xi_{i+1}^{i,d_i})^*$, and
$\gamma=\sum_{i=1}^{k-1}(\xi_{i}^{i+1,d_{i+1}})^*$ be as above.
We claim that 
$(\ff\alpha{\oplus}\ff\beta{\oplus}\ff\gamma)\cap(\gt g_e^*)_{\rm sing}=\{0\}$.
Indeed each non-zero point $x\alpha+y\beta$ is regular by 
\cite[Proposition 3.3]{ppy}. In order  to prove that 
$\gamma+x\alpha+y\beta$ is regular for all $x,y\in\ff$, we
use the action $\rho$ of $\ff^*$ defined by Formula~(\ref{ro}). 
Direct calculation shows that $\rho(t)(\gamma+x\alpha+y\beta)=
\gamma+xt\alpha+yt^2\beta$. Since 
$\gamma=\lim_{t\to 0}\rho(t)(\gamma+x\alpha+y\beta)$ and it is 
regular by Lemma~\ref{gamma}, all points $\rho(t)(\gamma+x\alpha+y\beta)$,
including $\gamma+x\alpha+y\beta$, are regular.

The  result follows, since  the subset  $(\gt g_e^*)_{\rm sing}$ 
is conical and Zariski closed.  
\end{proof}

 Let us say that a subalgebra
${\mathcal A}$ is {\it Poisson-commutative} 
if $\{{\mathcal A},{\mathcal A}\}=0$. 
Our main interest in the ``codim 3'' property is 
motivated by some application related to Poisson-commutative subalgebras 
of ${\mathcal S}(\gt g_e)$.

\begin{df}(Panyushev) A Lie algebra $\gt q$ is said to be 
{\it $n$-wonderful} if 
\begin{itemize}
\item[\sf (i)] \ 
${\mathcal S}(\gt q)^{\gt q}=\mathbb F[H_1,\ldots, H_{\ind\gt q}]$ 
is a polynomial algebra in $\ind\gt q$ variables; 
\item[\sf (ii)] \ 
all $H_i$ are homogeneous and 
 $\sum\limits_{i=1}^{\ind\gt q} \deg H_i=\displaystyle\frac{\dim\gt q+\ind\gt q}{2}$;
\item[\sf (iii)] \ $\codim (\gt q^*_{\rm sing})\ge n$.
\end{itemize}
\end{df}

The centralisers in types $A$ and $C$ are $2$-wonderful by \cite{ppy}. 
Now we know that in type $A$ they are $3$-wonderful. 

For $a\in\gt q^*$ let $\partial_a$ be a linear 
differential operator (partial derivative) on 
${\mathcal S}(\gt q)$ such that 
$\partial_a\xi= a(\xi)$ on $\xi\in\gt q$.

\begin{thm}\cite{codim3}\label{max}
Suppose that $\gt q$ is $3$-wonderful and $a\in\gt q^*_{\rm reg}$. 
Let ${\eus F}_a\subset{\mathcal S}(\gt q)$ 
be a subalgebra generated by the partial derivatives 
$\partial_a^m H_i$ ($m\ge 0$, $1\le i\le\ind\gt q$). 
Then ${\eus F}_a$ is a polynomial algebra in 
$(\dim\gt q+\ind\gt q)/2$ variables and it is maximal 
(with respect to inclusion) 
Poisson-commutative subalgebra of ${\mathcal S}(\gt q)$. 
\end{thm}
 
Theorem~\ref{max} is applicable to the centralisers $\gt g_e$ in type $A$. 
Similar results concerning ${\eus F}_{\alpha}$ with 
$\alpha\in\gt g_e^*$ being slightly more general or the same 
as in Theorem~\ref{c3A} are recently obtained by A.\,Joseph. 

In type $C$ the picture is not so nice. There are nilpotent elements 
such that subalgebras ${\eus F}_a$ are never maximal. 

\begin{ex}\label{sp42} 
Let $e\in\cN(\gt{sp}_6)$ be defined by 
the partition $(4,2)$. (It was  considered in Example~\ref{sp42-com}.) 
Then $\dim[\gt g_e,\gt g_e]=2$, hence $\codim(\gt g_e^*)_{\rm sing}=2$. 
Let ${\eus F}_a$ be as in Theorem~\ref{max} with $a\in\gt g_e^*$. 
For this centraliser, ${\eus F}_a$  is never
maximal among Poisson-commutative subalgebras of ${\mathcal S}(\gt g_e)$. 
The general construction of \cite{ppy} allows us to write down the invariants.
They are $H_1=\xi_1^{1,1}+\xi_2^{2,1}$, 
$H_2=\xi_1^{1,3}$, and $H_3=4\xi_1^{1,3}e_2+\eta\eta$, 
with $\eta=\xi_1^{2,1}-\xi_1^{2,3}$. %% similar to Example~\ref{sp42-com}. 
If $a$ is not regular, i.e., $a$ is zero on 
$[\gt g_e,\gt g_e]=\left< \xi_1^{1,3},\eta\right>_\ff$,
then $\partial_a H_3$ is proportional to  $\xi_1^{1,3}=H_2$ and
${\eus F}_a={\mathcal S}(\gt g_e)^{\gt g_e}$ is not maximal.

Assume that $a\in(\gt g_e^*)_{\rm reg}$. 
Then ${\eus F}_a$ is generated by four elements, the 
invariants $H_i$ and $x=\partial_a H_3$, which is an element of 
$(\gt g_e)_a$. %%% or rather by $(\gt g_e)_a$ and $H_3$.
According to Example~\ref{sp42-com}, $\gt g_e$ satisfies 
condition~ (\ref{strange}), hence $x$ is not regular, i.e.,
$\dim(\gt g_e)_x>3$. 
Clearly $(\gt g_e)_x$ commutes with ${\eus F}_a$, but is not contained 
in it.   Therefore  ${\eus F}_a$ is not maximal.
 \end{ex}
 
It is quite possible that there are some wide classes of nilpotent 
elements in type $C$ for which "codim 3" condition holds. For example, 
it is satisfied for  nilpotent elements given by partitions $(d^k)$
with odd $d$ and even $k$. By the contrast, it is not satisfied 
for partitions $(d^k)$ with even $d$ and  $k>1$. 

%%%%%%%%%%%%%%%%%%%%%%%%%%%%%%%%%%%%%%%%%%%%%%%%%%%%%
%%%%%%%%%% SECTION %%%%%%%%%%%%%%%%%%%%%%%%%%%%%%%%%%
%%%%%%%%%%%%%%%%%%%%%%%%%%%%%%%%%%%%%%%%%%%%%%%%%%%%%
\section{Explicit formulas for  symmetric invariants 
        of centralisers in type $A$}\label{forms}
%%%%%%%%%%%%%%%%%%%%%%%%%%%%%%%%%%%%%%%%%%%%%%%%%%%%%

In types $A$ and $C$ algebras of symmetric invariants 
${\mathcal S}(\gt g_e)^{\gt g_e}$ were described 
in \cite{ppy}. The outline of that approach is given in 
Section~\ref{str}. In type $A$ we have an alternative description 
of 
%%%%the algebra of symmetric invariants 
${\mathcal S}(\gt  g_e)^{\gt g_e}$ suggested by Brown and Brundan  
\cite{bb}. They reproved that this algebra is a polynomial algebra in 
$\rk\gt g$ variables.  
Comparing the approaches of \cite{bb} and \cite{ppy} 
we confirm \cite[Conjecture 4.1]{ppy}.

Brown and Brundan used different notation. 
%%% Fortunately, one can check that elements 
At first we should reinterpret symbols
$e_{i,j;r}$ introduced in  \cite{bb} in terms of $\xi_i^{j,r}$. 
According to \cite[Formula (1.1)]{bb}, 
$e_{i,j;r}$ is a sum of matrix units 
$e_{h,k}$, where $w_h$ is a basis
vector of $\VV[i]$ and $w_k$ is a basis 
vector of $\VV[j]$, in the notation of 
Section~1 of the present paper. Thus 
$e_{i,j;r}\in\Hom(\VV[j],\VV[i])$. 
More precisely, $e_{i,j;r}$ is a sum 
of the matrix units on the (above)
diagonal line in the $i,j$-rectangular, see Picture~\ref{pikcha_Arnold}.
Hence $e_{i,j;r}=\xi_j^{i,s}$ for some $s$.
In order to calculate $s$, note that if $r=\lambda_j-1=d_j$,
then $s=d_i$ and for $r=\lambda_j-\min(\lambda_i,\lambda_j)$
we get $s=d_i-\min(d_i,d_j)$.
The final answer is that 
$e_{i,j;r}=\xi_j^{i,s}$ with  $s=r+d_{i}-d_{j}$. 
 
The cardinality of a finite set $I$ is denoted by $|I|$.
Given a permutation $\sigma$ of a subset 
$I=\{i_1,\ldots,i_m\}\subset\{1,\ldots,k\}$ and a nonnegative function $\bar
s\colon\,I\to\mathbb Z_{\ge 0}$, we associate with the triple
$(I,\sigma, \bar s)$ the monomial 
$$
\Xi(I,\sigma,\bar s)\,:=\,\,\xi_{i_1}^{\sigma(i_1),\,\bar
s(i_1)}\xi_{i_2}^{\sigma(i_2),\,\bar s(i_2)}
    \ldots\xi_{i_m}^{\sigma(i_m),\,\bar s(i_m)}\in {\mathcal S}(\gt g_e)
$$
of degree $m=|I|$. If $\bar s(i_j)$ does not satisfies the restriction
on $s$ given in Section~1, then we assume that
$\xi_{i_j}^{\sigma(i_j),\bar s(i_j)}=0$.
For every $\Xi=\Xi(I,\sigma,\bar s)$ we denote by
$\lambda(I,\sigma,\bar s)$ the weight of $\Xi$ with respect to 
$h$, where $h$ is a characteristic of $e$.
Obviously, $\lambda(I,\sigma,\bar s)$ is the sum of the 
$\ad h$-eigenvalues ($h$-weights)
%%%$\lambda\big(\xi_{i_j}^{\sigma(i_j),\bar s(i_j)}\big)$
of the factors
$\xi_{i_j}^{\sigma(i_j),\bar s(i_j)}$.

Suppose that $\gt g=\gt{gl}(\VV)$.
Let $\{\Delta_1,\ldots \Delta_{\rk\g}\}$ be a generating set in 
$\mathbb F[\gt g]^{\gt g}$ such that $\Delta_i(\xi)$ are coefficients of the 
characteristic polynomial of $\xi\in\gt g$.  Identifying $\gt g$ and 
$\gt g^*$ we identify also $\mathbb F[\gt g]^{\gt g}$ and 
${\mathcal S}(\gt g)^{\gt g}$. 
Let $\{F_i\}$ be the corresponding (to $\{\Delta_i\}$) set of generators 
of ${\mathcal S}(\gt g)^{\gt g}$. By a result of 
\cite{ppy}, the ${^e\!}F_i$'s form a generating set of 
${\mathcal S}(\gt g_e)^{\gt g_e}$. The following statement 
was conjectured to be true in \cite{ppy}. It will be proved in 
this section. 

\begin{thm}\label{explicit} 
Let $1\le\ell\le\rk\gt g$ and set $m:=\deg{^e\!}F_\ell$. Then 
up to a non-zero constant,
$$
{^e\!}F_\ell\,=\,\,\,\sum\limits_{|I|\,=\,m,\,\,
\lambda(I,\sigma,\bar s)\,=\,2(\ell-m)} 
({\rm sgn}\,\sigma)\,\Xi(I,\sigma,\bar s),
$$
where the summation is taken over all subsets $I$,
all permutations $\sigma$ of $I$, and over all functions 
$\bar s$.
\end{thm} 

\begin{lm}\label{weights} In the above notation we have 
$\lambda(I,\sigma,\bar s)=2\sum\limits_{j\in I} \bar s(j)$.
\end{lm}
\begin{proof} It is not difficult to compute that the weight of 
$\xi_i^{j,s}$ is equal to $2(d_i-d_j+s)$. Therefore 
$$
\lambda(I,\sigma,\bar s)=2\sum_{j\in I} (d_j-d_{\sigma(j)}+\bar s(j))=
 2\sum_{j\in I} \bar s(j).
$$
The second equality holds because $\sigma$ is a permutation. 
\end{proof}

\begin{comment}
\begin{rmk} 
According to \cite[Lemma~3.8]{bb}, 
for each monomial $\Xi(I,\sigma,\bar s)$,
satisfying condition of Theorem~\ref{explicit},
all $\bar s(i_j)$ belong to the ranges given in Section~1. 
In other words, all elements $\xi_{i_j}^{\sigma(i_j),\bar s(i_j)}$   
are indeed vectors of $\gt g_e$ and 
$\Xi(I,\sigma,\bar s)$ is never zero.
\end{rmk}
\end{comment}

Set $\tilde{\xi}_i^{j,s}:=\xi_i^{j,s}-\delta_{s,0}\delta_{i,j}(i-1)(d_i+1)$, 
where $\delta_{i,j}=1$ for $i=j$ and is zero otherwise.  
Note that $e_{i,i;0}=\xi_i^{i,0}$ and, as above, 
for a permutation $\sigma$ of $I$ we have 
$\sum\limits_{j\in I} (\bar s(j)+d_j-d_{\sigma(j)})=
  \sum\limits_{j\in I} \bar s(j)$. 
Taking these two facts into account,
we rewrite Formulas~(1.2) and (1.3) of \cite{bb} in 
the $\xi_j^{i,s}$-notation. 
For each set $I$ of indices $1\le i_1< i_2 < \ldots < i_m$ and
each permutation $\sigma$,
define 
$$\widetilde{\Xi}(I,\sigma,\bar s):=
 \,\,\tilde{\xi}_{i_1}^{\sigma(i_1),\,\bar
s(i_1)}\tilde{\xi}_{i_2}^{\sigma(i_2),\,\bar s(i_2)}
    \ldots\tilde{\xi}_{i_m}^{\sigma(i_m),\,\bar s(i_m)}\in {\bf U}(\gt g_e).
$$
Let $\ell$ be in the range $1\le\ell\le\rk\gt g$ 
and $m=\deg{^e\!}F_\ell$. In view of Lemma~\ref{weights}, 
we can express elements $z_\ell$ of \cite{bb} as follows 
\begin{equation}\label{bb2}
z_\ell=\sum_{|I|=m,\,\, \lambda(I,\sigma,\bar s)=2(\ell-m)}
 ({\rm sgn}\,\sigma)\,\widetilde{\Xi}(I,\sigma,\bar s),
\end{equation}
where the summation is taken over all subsets $I$,
all permutations $\sigma$ of $I$, and over all functions 
$\bar s$.

The main theorem of \cite{bb} states that the elements $z_\ell$ generate the centre 
of ${\bf U}(\gt g_e)$ and that their symbols,
elements of ${\mathcal S}(\gt g_e)$,  
denoted $\overline{z_\ell}$,
are algebraically independent. 

\vskip0.5ex
{\it Proof of Theorem~\ref{explicit}.} \ 
In \cite{ppy} a slightly weaker statement was proved. 
More precisely, it was shown that
for each $\ell\le \rk\gt g$, we have
$$
{^e\!}F_\ell\,=\,\,\,\sum\limits_{|I|\,=\,m,\,\,\lambda(I,\sigma,\bar
s)\,=\,2(\ell-m)} a(I,\sigma,\bar s)\,\Xi(I,\sigma,\bar s)
$$
for some $a(I,\sigma,\bar s)\in\mathbb F$.
Here we prove that 
each ${^e\!}F_\ell$ is a non-zero  multiple of the 
symbol $\overline{z_\ell}$. 

Following Brown and Brundan, restrict the 
invariants to an affine 
slice $\eta+V\subset\gt g_e^*$. In our notation, 
$\eta=\sum_{i=1}^{k-1}(\xi_{i+1}^{i,d_i})^*$ and 
$V$ is the subspace generated by $(\xi_1^{i,s})^*$. 
According to \cite{bb}, this 
restriction map 
$\psi:\enskip {\mathcal S}(\gt g_e)^{\gt g_e}\to \ff[\eta+V]$ 
is an isomorphism.

Suppose that $\deg {^e\!}F_\ell=m$. Then both 
$\psi({^e\!}F_\ell)$ and $\psi(\overline{z_\ell})$
are proportional to $\xi_1^{m,s}$ with 
$s=\ell-(d_1+\ldots+d_{m-1})-m$. This completes the proof 
of Theorem~\ref{explicit}. 
\hfill{$\Box$}

\section{Fibres of the quotient morphism 
$\gt g_e^*\to\gt g_e^*/\protect\!\protect\!/ G_e$} %%%  in type $A$}

Suppose that  $\gt g$ is either of type $A$ or $C$.
Then ${\mathcal S}(\gt g_e)^{G_e}=\mathbb F[H_1,\ldots,H_{\rk\gt g}]$,
where $H_i={^e\!}F_i$ for a certain (good) generating 
set $\{F_i\}\subset{\mathcal S}(\gt g)$
of $\gt g$-invariants, 
see \cite{ppy}. In particular, the algebra of symmetric 
$G_e$-invariants is
finitely generated and we can consider the quotient morphism
$\gt g_e^*\to \gt g_e^*/\!\!/G_e$, where 
$\gt g_e^*/\!\!/G_e={\rm Spec}\,{\mathcal S}(\gt g_e)^{G_e}$ and 
each $x\in\gt g_e^*$ maps to 
$(H_1(x),\ldots,H_{\rk\g}(x))$. 
In this section we are interested in the fibres of the quotient 
morphism. By \cite[Section~5]{ppy}, in type $A$ all fibres of this morphism 
are of dimension $\dim\gt g_e-\rk\gt g$.

Consider a point $\alpha=\sum_{i=1}^{k} a_i(\xi_i^{i,d_i})^*\in\gt g_e^*$, 
where $a_i$ are pairwise distinct non-zero numbers and $\gt g=\gt{gl}(\VV)$.
As was already mentioned, it is a generic point and 
$\gt h=(\gt g_e)_\alpha$ is a generic stabiliser for the coadjoint 
action of $\gt g_e$. In case $e\in\gt{sp}(\VV)$,
similar statements remain true for the restriction 
of $\alpha$ to $\gt{sp}(\VV)_e$ and $\gt h\cap\gt{sp}(\VV)$, see \cite{fan}.
Set $H:=(\GL(\VV)_e)_\alpha$. %%% Clearly $\gt h=\Lie H$. 
Then $H$ is connected and $(\GL(\VV)_e)_\gamma$ 
is conjugate to $H$ whenever 
$(\gt{gl}(\VV)_e)_\gamma$ is conjugate to $\gt h$. 
In other words, $H$ is a generic stabiliser 
for the coadjoint action of $GL(\VV)_e$. 
Again, if $e\in\gt{sp}(\VV)$, then 
$H\cap Sp(\VV)$ is a generic stabiliser for the 
coadjoint action of $Sp(\VV)_e$. 
 
Recall that $\gt h=\left<\xi_i^{i,s}\right>$ and 
$\gt h$ containes a maximal torus 
$\gt t=\left<\xi_i^{i,0}\right>$ of $\gt{gl}(\VV)_e$. 
Thereby $H=T\ltimes U$, where $T$ is a maximal torus of 
$GL(\VV)_e$ and $U$ is contained in the unipotent radical of $GL(\VV)_e$.
Likewise, for $e\in\gt{sp}(\VV)$, the generic stabiliser 
$H\cap Sp(\VV)$ contains a maximal 
torus $T\cap Sp(\VV)$ of $Sp(\VV)_e$.  
Applying the following lemma, we get that 
generic coadjoint orbits of centralisers in types $A$ and $C$ 
are closed. 

\begin{lm}\label{orbit} Suppose that an algebraic group $G$ 
acts on an affine variety $X$ and a stabiliser $G_x$ of a 
point $x\in X$ contains a maximal torus $T$ of $G$.
%%% and is contained  in a Borel subgroup $B$ of $G$. 
Then the orbit $Gx$ is closed. 
\end{lm}
\begin{proof} Let us choose a Borel subgroup $B\subset G$ 
containing $T$. 
Then the $B$-orbit $Bx$ is closed, because 
it coincides with the orbit of a unipotent group,
in this case of the unipotent radical of $B$.

We have a closed subgroup $B\subset G$ such that 
the quotient $G/B$ is complete and the orbit $Bx$ is closed.
It follows that $G{\cdot}Bx=Gx$ is also closed, 
see e.g. \cite[Lemma~2 in Section~2.13]{stein}.
%%%
%%% Consider a direct product $G/B\times X$ 
% and its closed subset 
% $$
% G\times_{B} Bx=\{(gB,y) \mid g^{-1}y\in Bx\}.
% $$
% Let $\pi:\enskip G/B\times X\to X$ be the projection 
% on the second factor. Then $\pi(G\times_{B} Bx)=Gx$. 
% Since $G/B$ is complete, $Gx$ is closed. 
\end{proof}

\noindent
Lemma~\ref{orbit} is a well-known and classical fact. In case 
of complex reductive group $G$, similar result was 
proved by Kostant in 1963, see \cite[proof of Lemma~5]{kost63}.

\begin{thm}\label{fibres} 
If $\gt g$ is either $\gt{gl}(\VV)$ 
or $\gt{sp}(\VV)$, then a generic fibre of the quotient 
morphism $\gt g_e^*\to \gt g_e^*/\!\!/G_e$ consists of 
a single closed $G_e$-orbit. 
\end{thm} 
\begin{proof}
In both these cases the coadjoint action of 
$G_e$ has a generic stabiliser, which contains 
a maximal torus of $G_e$, see \cite[Section~4]{fan} and 
discussion before Lemma~\ref{orbit}. 
By Lemma~\ref{orbit}, generic orbits are closed. 
Since $\ind\gt g_e=\rk\gt g$, generic coadjoint 
$G_e$-orbits and generic fibres of the quotient morphism have the 
same dimension, $\dim\gt g_e-\rk\gt g$. 
Hence there is an open subset 
$U\subset\gt  g_e^*/\!\!/G_e$ such that the fibre 
over each $u\in U$ contains a closed $G_e$-orbit of 
maximal dimension and that orbit is an irreducible 
component of the fibre.   

In cases of our interest 
${\mathcal S}(\gt g_e)^{G_e}={\mathcal S}(\gt g_e)^{\gt g_e}$,
see  \cite[Theorems~4.2 and 4.4]{ppy}.
Hence each element of 
${\mathcal S}(\gt g_e)$, which is algebraic over 
${\rm Quot}\,(\ff[\gt g_e^*]^{G_e})$, is $\gt g_e$- and $G_e$-invariant. 
This means that
${\mathcal S}(\gt g_e)^{G_e}$ is algebraically closed in 
${\mathcal S}(\gt g_e)$. 
By Theorem~\ref{fibre}, proved in the appendix, 
generic fibres of the quotient morphism are connected.
Shrinking $U$ if necessary, 
we may assume that the fibres over elements of $U$ are 
connected. Then each of them consists of a single 
closed $G_e$-orbit of maximal dimension.
\end{proof} 
 
\noindent
Theorem~\ref{fibres}
was proved in a discussion with A.\,Premet during his visit to the 
Max-Planck Institut f\"ur Mathematik (Bonn) 
in Spring 2007. 
 
\begin{rmk} The proof of Theorem~\ref{fibres}
can be completed in a slightly different way.
The ring $\mathbb F[\gt g^*]$ is a unique factorisation 
domain. If $a\in{\mathcal S}(\gt g_e)^{\gt g_e}$, 
then all prime factors of $a$ are also $\gt g_e$-invariant. 
One can show quite elementary that 
the field ${\rm Quot}\,(\ff[\gt g_e^*]^{G_e})$
is algebraically closed in $\mathbb F(\gt g^*)$. 
Then generic fibres are known to be irreducible,
see e.g. \cite[Chapter~2, Section~6.1]{Shaf}.
\end{rmk}
 
\begin{rmk} If $\gt g$ is of type $A$ or $C$, then,
as was mentioned above, 
the coadjoint action of 
$G_e$ has a generic stabiliser, which contains 
a maximal torus of $G_e$. This means that 
the ring of semi-invariants 
${\mathcal S}(\gt g_e)^{\gt g_e}_{\rm si}$ 
coincides with ${\mathcal S}(\gt g_e)^{\gt g_e}$.
Lie algebras $\gt q$ with 
${\mathcal S}(\gt q)^{\gt q}_{\rm si}$
being a polynomial ring are actively studied,
see e.g. \cite{oomsb}. In particular, if 
$\gt g$ is of type $A$ and $\gt g_e$ is non-Abelian,
then \cite[Proposition 1.6]{oomsb} combined with 
Theorem~\ref{c3A}, implies that each irreducible component
of $(\gt g_e^*)_{\rm sing}$ has dimension $\dim\gt g_e{-}3$.   
\end{rmk} 
 
\vskip0.7ex

In contrast with a generic fibre, 
the null-cone $\cN(e)$ (the fibre containing zero) 
may have infinitely many closed orbits and 
there might be no regular elements 
(and hence no open orbits) in some of its components. 
Dealing with $\cN(e)$, 
we will freely use the precise formulas for the generators
$H_i={^e\!}F_i$, obtained in Section~\ref{forms}. 

\begin{ex} Let $e\in\cN(\gt{gl}_6)$ be  given by 
the partition $(4,2)$. Here $\dim\gt g_e-\rk\gt g=4$, hence 
all irreducible components of $\cN(e)$ are of dimension $4$.
There are $4$ elements in the centre of 
$\gt g_e$, they are linear invariants $H_1,H_2,H_3,H_4$. The other  
 two  invariants $H_5$ and $H_6$ 
are of degree $2$. Until the end of the example, we replace 
$\gt g_e^*$ by a subspase $P\subset\gt g_e^*$ 
defined by $H_1=\ldots=H_4=0$ and regard $\cN(e)\subset P$ as the
zero set of $H_5$ and $H_6$. 

Then restricted to $P$, the invariants $H_5$ and $H_6$ 
are expressed by the formulas    
$H_6=\xi_1^{2,1}\xi_2^{1,3}$ and 
$H_5=\xi_1^{2,1}\xi_2^{1,2}+\xi_1^{2,0}\xi_2^{1,3}$.  
Both are zero on the linear subspace defined by 
$\xi_1^{2,1}=\xi_2^{1,3}=0$. 
Hence a four-dimensional 
vector space $X\subset P$ %%\gt g_e^*$ 
generated by vectors  
$$
(\xi_1^{1,0})^*-(\xi_2^{2,0})^*,  (\xi_1^{1,1})^*-(\xi_2^{2,1})^*, 
    (\xi_1^{2,0})^*, (\xi_2^{1,2})^*
$$
is an irreducible component of the null-cone $\cN(e)$. 
The action of $G_e$ on $X$ has a $7$-dimensional 
ineffective kernel. Since coadjoint orbits are 
even-dimensional, $G_e$-orbits on $X$ are either trivial 
or $2$-dimensional. Essentially the only non-trivial actions 
are: 
$$
\begin{array}{l}
\ad^*(\xi_1^{1,0}-\xi_2^{2,0}) (\xi_1^{2,0})^* = (\xi_1^{2,0})^*, \ \
\ad^*(\xi_1^{1,0}-\xi_2^{2,0}) (\xi_2^{1,2})^* = -(\xi_2^{1,2})^*, \\[0.6ex]
\text{ and }\  \  
  -\ad^*(\xi_1^{2,0})(\xi_1^{2,0})^*=\ad^*(\xi_2^{1,2})(\xi_2^{1,2})^*=
 (\xi_1^{1,0})^*-(\xi_2^{2,0})^*. \\
\end{array}
$$
Thus $X$ contains a $2$-parameter family of closed 
$2$-dimensional $G_e$-orbits; two non-closed $2$-dimensional orbits; and 
a $2$-parameter family of $G_e$-invariant points. 
In particular, $X$ contains no regular elements.

%%% This example is particularly bad, because here 
For this nilpotent element 
the ideal 
$I=(\cS(\gt g_e)^{\gt g_e}_{\circ})\lhd{\mathcal S}(\gt g_e)$ generated by 
the homogeneous  invariants of positive degree 
is not radical. After restriction to $P$, where $I$ is generated by 
$H_5$ and $H_6$,  
we have $\xi_1^{2,1}\xi_2^{1,2}\not\in I$, but 
$$
(\xi_1^{2,1}\xi_2^{1,2})^2=\xi_1^{2,1}\xi_2^{1,2} H_5-\xi_1^{2,0}
     \xi_2^{1,2} H_6 \in I.
$$
\end{ex}

A very interesting problem is to describe the irreducible components of
$\cN(e)$ in type $A$.
Here we compute the number of these components in two particular cases.

\begin{lm}\cite[Theorem 1.2]{Dima} \label{z2-deg}
Suppose that $\gt q$ is a Lie algebra such that
${\rm codim}\,\gt q^*_{\rm sing}\ge 2$ and 
$H_1,\ldots,H_{\rk\gt q}$ are algebraically independent homogeneous elements 
of ${\mathcal S}(\gt q)^{\gt q}$ with 
$\sum\limits_{i=1}^{\rk\gt q} \deg H_i=({\dim\gt q+\rk\gt q})/{2}$. 
Then $H_1,\ldots,H_{\rk\gt q}$ generate the whole algebra 
${\mathcal S}(\gt q)^{\gt q}$ 
of symmetric $\gt q$-invariants.   
\end{lm}

\begin{prop} Suppose that $e\in\gt{gl}_{m+n}$ is defined by the partition 
$(n,1^m)$ with $n\ge 2$. Then $\cN(e)$ has $m+1$ irreducible components.
\end{prop}
\begin{proof} 
Let $P\subset\gt g_e^*$ be the zero-set of 
linear invariants. Then $P$ is isomorphic to the dual
space of a Lie algebra $\gt q=\gt{gl}_m\ltimes V$, where 
$V\cong \ff^m\oplus(\ff^m)^*$ is a commutative ideal. 
Note that $\gt q$ is a quotient 
of $\gt g_e$ and $\gt g_e$ acts on $P$ via the coadjoint representation 
of $\gt q$. 

Set $L=GL_m$ and $\gt l=\Lie L$.
Identifying $\gt l^*$ with the annihilator ${\rm Ann}(V)\subset\gt q^*$ and 
$V^*$ with  ${\rm Ann}(\gt l)\subset\gt q^*$, we consider 
$\gt l^*$ and $V^*$ as subspaces of $\gt q^*$ and of 
$\gt g_e^*$. Take $H_i={^e\!}F_i$ with $i>n$. 
Then $\deg H_i=i{-}n{+}1$ and the restriction ${H_i}\vert_{P}$ 
is a bi-homogeneous polynomial in 
variables $\gt l$ and $V$ of bi-degree $(i{-}n{-}1,2)$.

The image of the projection $\cN(e)\to V^*$ coincides 
with the zero set $\cN(V)$ of ${H_{n+1}}\vert_{P}$. 
There are four $L$-orbits in $\cN(V)$: the 
open orbit,  zero, and two intermediate, 
in $(\ff^m)^*$ and $\ff^m$. Note that 
the subsets $\gt l^*\oplus(\ff^m)^*$ and
$\gt l^*\oplus\ff^m$ of $\gt g_e^*$ are defined by the equations 
$\xi_1^{1,t}=0$ ($t=0,\ldots,m-1$) and 
$\xi_1^{i,1}=0$  or $\xi_i^{1,m-1}=0$, respectively  (here $i>1$). 
Explicit formulas exhibited in Section~6 show that
both these subspaces are contained in $\cN(e)$. Since they are irreducible 
and of the right dimension, $\dim\gt g_e-(m{+}n)$, they are 
irreducible components of $\cN(e)$. 

Let $X$ be an irreducible component of $\cN(e)$ distinct from either  
$\gt l^*\oplus(\ff^m)^*$ or $\gt l^*\oplus\ff^m$.
Then the image of the projection $X\to V^*$ is either
zero or contains an open $L$-orbit ${\mathcal O}$.      
The first case is not possible because $\dim{\gt l^*}< \dim\cN(e)$. 
Thus, it remains to  deal with the irreducible components of 
the intersection $\cN(e)\cap (\gt l^*{\times}{\mathcal O})$. 
Since $G_e$ is connected, each irreducible component of $\cN(e)$ is 
$G_e$-invariant and the problem reduces to 
the intersection $\cN(e)\cap (\gt l^*{\times}\{v\})$,  where 
$v\in{\mathcal O}$. Since $V$ is a commutative ideal of $\gt q$,
it acts on the fibre 
$\gt l^*{\times}\{v\}$. This action of $V$ has a slice 
$S\subset\gt l^*{\times}\{v\}$,
isomorphic to $\gt l_v^*{\times}\{v\}$, which meets each $V$-orbit exactly 
once, see e.g. \cite[Lemma~4]{int}.  
Since both $L_v$ and $V$ are connected, 
$\cN(e)\cap (\gt l^*{\times}{\mathcal O})$ has exactly the same number 
of irreducible components as the zero-set of 
${H_i}\vert_{S}$. 

The restrictions of ${H_i}$ with $n{+}2\le i\le n{+}m$ to ${S}$
are algebraically independent, otherwise $\cN(e)$ 
would have a component of dimension 
$(\dim\gt g_e-\rk\gt g)+1$. Identifying $S$ with $\gt l_v^*$
we may consider them as 
$\gt l_v$-invariant elements of ${\mathcal S}(\gt l_v)$. 
One readily computes that 
$\gt l_v\cong (\gt{sl}_{m})_{\hat e}$, where 
$\hat e$ is a nilpotent element 
defined by the partition $(2,1^{m-2})$. 
Clearly $\deg({H_i}\vert_{S})=\deg H_i-2=n{-}1$ for $i>n$. 
Therefore we get $m{-}1=\ind\gt l_v$ polynomials 
of degrees $1,2,\ldots,m{-}1$. The sum of degrees is equal 
to $(\dim\gt l_v+\ind\gt l_v)/2$. 
There is no 
consequential difference between centralisers in 
$\gt{gl}_m$ and $\gt{sl}_m$. Therefore, according to 
Theorem~\ref{c3A}, the codimension 
of $(\gt l_v^*)_{\rm sing}$ is grater than $2$. 
Thus all conditions of Lemma~\ref{z2-deg} are satisfied. 
Hence ${H_i}\vert_{S}$ generate ${\mathcal S}(\gt l_v)^{\gt l_v}$ 
and $\cN(e)\cap S$ is isomorphic to the null-cone $\cN(\hat e)$
associated with the nilpotent element $\hat e\in\gt{gl}_m$.  

If $m=0$, then $\cN(e)$ is irreducible. 
For $m=1$ there are two irreducible components, 
since ${H_{n+1}}\vert_{P}=\xi_1^{2,0}\xi_2^{1,n{-}1}$.
Arguing by induction on $m$, 
the may assume that $\cN(\hat e)$ has $m{-}1$ components. 
Then $\cN(e)$ has $m{-}1{+}2=m{+}1$ components. 
\end{proof}

\begin{prop} Suppose that $e\in\gt{gl}_{n+m}$ is defined by the partition 
$(n,m)$ with $n\ge m$. 
Then $\cN(e)$ has $\min(n-m,m)+1$ irreducible components.
\end{prop}
\begin{proof} 
Again we replace $\gt g_e^*$ be the zero set  $P\subset\gt g_e^*$
of the linear $\gt g_e$-invariants. 
Suppose first that $m\le n-m$. Set $x_i:=\xi_1^{2,m-i}$ and $y_i:=\xi_2^{1,n-i}$ for 
$1\le i\le m$. Then $\cN(e)$ is defined by the polynomials 
$f_q=\sum\limits_{i+j=q} x_iy_j$ with $2\le q\le m+1$. Each irreducible components is 
given by  a partition $m=a+b$, where $a,b\ge 0$. 
It is a linear subspace defined by 
$x_1=\ldots=x_a=0$, $y_1=\ldots=y_b=0$. 
Hence there are exactly $m+1$ components.  

Consider now the second case, there $n-m<m$. Set $k:=n-m$. Retain the notation for 
$x_i$ and $y_i$. Set in addition $z_i:=\xi_2^{2,m-i}$. Then the 
restrictions of non-linear symmetric invariants $H_i$ to $P$ are 
given by the polynomials 
$$
\begin{array}{l}
f_q=\sum\limits_{i+j=q} x_iy_j  \ \text{ with } \  2\le q\le k+1; \\
\qquad \qquad 
 \text{ and } 
 f_p=\sum\limits_{i+j=p} x_iy_j+\sum\limits_{i+j=p-k}z_iz_j \ \text{ with } \  
k+2\le p\le m+1. \\
\end{array}
$$ 
For example, here  $f_{k+2}=x_1y_{k+1}+\ldots x_{k+1}y_1+z_1^2$ and
$f_{k+3}=x_1y_{k+2}+\ldots x_{k+2}y_1+2z_1z_2$. 
Note that variables $z_j$ appear in these equations only for $j\le m-k$.
The first equations, $f_q$, 
give rise to $k+1$ irreducible components, each of which is a linear subspace.  
Take one of these components, defined by 
$x_1=\ldots=x_a=0$, $y_1=\ldots=y_b=0$ with $a+b=k$, and 
let $P_{a,b}$ be the  intersection of this linear subspace with 
$\cN(e)$.  We are going to show that $P_{a,b}$   
is irreducible and that these components do not coincide 
for distinct partitions $k=a+b$.  

Let $P_{a,b}^{\circ}$ be a subset of $P_{a,b}$, where  
$z_1\ne 0$. Then $P_{a,b}^{\circ}$ is irreducible, because it is
defined by the equations $x_{a+1}y_{b+1}=-z_1^2$ and 
$z_j=f_{k+1+j}(\overline{x},\overline{y},z_2,\ldots,z_{j-1})/z_1$
for $2\le j\le m{-}k$ .  Note that 
$\dim P_{a,b}^{\circ}=\dim\gt  g_e-(m+n)-\dim \cN(e)$. 
On the complement 
$P_{a,b}\setminus P_{a,b}^{\circ}$ we have $z_1=0$
and equations $f_q=0$ and  $f_p=0$ reduce to the following 
\begin{equation}\label{dwa}
\begin{array}{l}
x_1=\ldots =x_a=y_1=\ldots = y_b=0, \\
x_{a+1}y_{b+1}=0,  \qquad  x_{a+1}y_{b+2}+x_{a+2}y_{b+1}=0, \\[0.5ex]
f_p=\sum\limits_{i+j=p} x_iy_j+ (z_2z_{p-k-2}+\ldots + z_{p-k-2}z_2)=0
   \enskip  \ \ \   \text{ for } \  k+4\le p\le m+1. \\
\end{array}
\end{equation}
Equations~(\ref{dwa})  are very similar to the original $f_p$'s and $f_q$'s. 
Using induction on $k-m$ and the previous case, where  
$n-m\ge m$, one can say that they define three irreducible components 
$P_{a+2,b}(\bar e)$, $P_{a+1,b+1}(\bar e)$, $P_{a,b+2}(\bar e)$
of the null-cone associated with a nilpotent element $\bar e$ 
with Jordan blocks $(n+2,m)$. One thing, 
which we should keep in mind, is that for $\bar e$  variables 
$z_2,\ldots z_{m-k-1}$ are used instead of $z_1,\ldots,z_{m-k-2}$. 
Since $\dim\cN(e)=\dim\cN(\bar e)$,  
the complement $P_{a,b}\setminus P_{a,b}^{\circ}$ 
is an irreducible subset of dimension $\dim\cN(e)-1$. In particular, 
it could not be a component of  $\cN(e)$ 
and we have proved that $P_{a,b}$ is an irreducible component.

Suppose that $a'+b'=k$ and $a'\ne a$. Then either $a'>a$ or $b'>b$. Anyway, if
$P_{a',b'}=P_{a,b}$, then $x_{a+1}y_{b+1}$ is zero on $P_{a,b}$. 
Hence $z_1$ is also zero on it. A contradiction, since we know that 
$z_1\ne 0$ defines a non-empty open subset $P_{a,b}^{\circ}\subset P_{a,b}$.
\end{proof}

There should be a combinatorial formula for the 
number of components.  Unfortunately, 
we do not have enough information even to make a 
conjecture. Apart from two cases considered in this section, little 
is known. If the partition is rectangular, i.e., all Jordan blocks are 
of the same size, then $\gt g_e$ is a Takiff Lie algebra and 
the null-cone is irreducible, see \cite[Appendix]{mus}.
A direct calculation shows that the number of irreducible components 
for the partition $(3,2,1)$ is $4$.

%%%%%%%%%%%%%%%%%%%%%%%%%%!!!!!!!!!!!!!!!!!!!!!%%%%%%%%%%%%%%%%%
%%%%!!!!!!!!!!!!!!!!!!!!! SECTION !!!!!!!!!!!!!!!!!!!!!!!!!%%%%%
%
\section{Further results on the null-cone}
%%Nice even elements: null-cone and symmetric invariants}
%
%%%%%%%%%%%%%%%%%%%%%%%%%%%%%%%%%%%%%%%%%%%%%%%%%%%%%%%%%%%%%%%%

Suppose that  $\gt g\subset\gt{gl}(\VV)$ is either 
$\gt{sp}(\VV)$ or $\gt{so}(\VV)$ and $e\in\gt g$ is such that $i'=i$ for all $i$
(in terms of Lemma~\ref{restr}). 
Here we prove that each irreducible component of 
$\cN(e)$ has dimension $\dim\gt g_e-\rk\gt g$. 
Similar result was obtained in \cite[Section 5]{ppy} for all nilpotent elements in 
$\widetilde{\gt g}=\gt{gl}(\VV)$. Our proof uses the same strategy. 

For $m\in\{1,\ldots,k\}$,  partition the set $\{1,\ldots, m\}$
into pairs $(j,m-j+1)$. If $m$ is odd, then there will be a
``singular pair'' in the middle consisting of the singleton
$\{(m+1)/2\}$.  Let  $V_m$ denote the subspace of 
$\widetilde{\gt g}_e$ spanned
by all $\xi_i^{j,s}$ with $i+j=m+1$, and set 
$V:=\bigoplus_{m\ge 1} V_m$. 
 Using the basis $\{(\xi_i^{j,s})^*\}$ of $\widetilde{\gt g}_e^*$ dual to
the basis $\{\xi_i^{j,s}\}$, we shall regard the dual spaces $V_i^*$
and $V^*$ as subspaces of $\widetilde{\gt g}_e^*$.

Since $i'=i$ for all $i$, the restriction 
of the $\gt g$-invariant form on $\VV$ to each 
$\VV_i$ is non-degenerate. Hence the partition  
into pairs $(j,m-j+1)$ can be pushed down to $\gt g_e$. 
Each $V_m$ is preserved by $\sigma$, where $\sigma$ is an involution 
of $\widetilde{\gt g}$ with $\gt g=\widetilde{\gt g}^\sigma$.
Let  $\widetilde{\gt g}=\gt g\oplus\widetilde{\gt g}_1$
be the corresponding symmetric decomposition. 
Let us identify $\gt g_e^*$ with the annihilator 
of $\widetilde{\gt g}_{1,e}$ in $\widetilde{\gt g}_e^*$.
Then the expressions $V_{\gt g,m}^*:=V_m^* \cap \gt g_e^*$
make sense and $V_{\gt g,m}^*=(V_m^*)^\sigma$, similarly set
$V_{\gt g}^*:=V^*\cap\gt g^*$. Note also 
that 
$$
\bar{\gt g}:=\gt g\cap \gt{gl}\big(\VV[1]{\,\oplus}\cdots{\oplus}\VV[k{-}1]\big)
$$
is a semisimple subalgebra of $\gt g$, either 
$\gt{so}\big(\VV[1]{\,\oplus}\cdots{\oplus}\VV[k{-}1]\big)$ or 
 $\gt{sp}\big(\VV[1]{\,\oplus}\cdots{\oplus}\VV[k{-}1]\big)$, depending 
on $\gt g$. 
Likewise $\gt g_k:=\gt g\cap\gt{gl}(\VV_k)$ is either 
$\gt{so}(\VV_k)$ or $\gt{sp}(\VV_k)$. 

Set $n:=\dim\VV$.  Let 
$\Delta_i\in\mathbb F[\widetilde{\gt g}]^{\widetilde{\gt g}}$ 
(with $1\le i\le n$) be the coefficients of the characteristic polynomial. 
Unlike Section~\ref{forms}, here we  consider  $\Delta_i$ as elements
of ${\mathcal S}(\widetilde{\gt g})$. Set $F_i:={\Delta_{2i}}\vert_{\gt g^*}$ 
for all $i$ in the range  $1\le i\le\rk\gt g$.
%%%
%%% except $i=n/2$ in case $n$ is even and 
%% $\gt g=\gt{so}(\VV)$, in that exceptional situation 
%% $F_{\rk\gt g}$ is a square root of  ${\Delta_{n}}_{\vert \gt g^*}$. 
Note that all $\Delta_i$ with odd $i$ are zero on $\gt g^*$. 
As was proved in \cite[Theorem~4.2 and Lemma~4.5]{ppy}, the polynomials ${^e\!}F_i$ 
are algebraically independent and in the symplectic case they 
 generate ${\mathcal S}(\gt g_e)^{\gt g_e}$. 
Let $\cN_F(e)\subset\gt g_e^*$ be the zero set of the polynomials 
${^e\!}F_i$. 

\begin{thm}
Suppose that $\gt g$ and $e\in\gt g$ satisfy the assumptions of this section. 
Then 
there exists a linear subspace 
$W_{\gt g}=\bigoplus_{m\ge 1} W_{\gt g,m}$ in 
 $V_{\gt g}^*$ of dimension $\rk\gt g$  such that 
$W_{\gt g,m}\subset V_{\gt g,m}^*$ for all $m$ and
$W_{\gt g}\cap\cN_F(e)=\{0\}$.
\end{thm}
\begin{proof}
We argue by induction on $k$. If $k=1$, then $e$ is a regular nilpotent element, 
all ${^e\!}F_i$ are linear functions and they form a basis of $\gt g_e$. Hence 
$\cN_F(e)=\{0\}$ and there is nothing to prove. Assume that
$k>1$ and for all
$k'<k$ the statement is true. 

Regard
the dual spaces $\bar{\g}^*$ and $\g_k^*$ as subspaces of $\g^*$.
Note that $e=e_k+\bar{e}$ where $e_k$ and $\bar{e}$ are the
restrictions of $e$ to  $\VV[k]$ and
$\VV[1]{\,\oplus}\cdots{\oplus}\VV[k{-}1]$, respectively. Clearly, $e_k$
is a regular nilpotent element in $\g_k$ and
$\bar{e}\in\bar{\g}$ is a nilpotent
element with Jordan blocks of sizes $d_1+1,\ldots, d_{k-1}+1$. 
Note that $V_{\gt g,m}^*\subset(\bar{\gt g}_{\bar e})^*$ for $m<k$. 

The restriction of ${^e\!}F_i$  (with $1\le 2i\le n-d_k-1$) 
to $(\bar{\gt g}_{\bar e})^*$ can be obtained 
as follows:  first restrict $\Delta_{2i}$ to the dual 
of $\gt{gl}\big(\VV[1]{\,\oplus}\cdots{\oplus}\VV[k{-}1]\big)$, getting again 
a coefficient of the characteristic polynomial, then restrict it further to  
$\bar{\gt g}$ and apply the ${^{\bar e}\!}F$-construction. 
Hence by the inductive hypothesis there is a subspace 
$\overline{W}_{\bar{\gt g}}=\,\bigoplus_{m=1}^{k-1}\, W_{\gt g,m}$
with $W_{\gt g,m}\subset V^*_{\gt g,m}$
such that $\dim\overline{W}_{\bar{\g}}=\rk\bar{\gt g}$ and
$\overline{W}_{\bar{\g}}\cap\cN_F(\bar{e})=\{0\}$.

Consider the remaining invariants. 
For $0\le q\le d_k{+}1$ set $\hat\varphi_{n-q}:={{^e\!}\Delta_{n-q}}\vert_{V^*}$. 
By \cite[Lemma~5.1]{ppy}, each $\hat\varphi_{n-q}$ is an element of 
${\mathcal S}(V_k)$. 
Let $X\subset V^*_k$ be the zero locus of the $\hat\varphi_\ell$
with $n\ge\ell \ge n{-}d_k{-}1$. 
Note that 
$$
\cN_F(e)\cap (\overline{W}_{\bar{\g}}\oplus V_{\gt g,k}^*)=
  (\cN_F(\bar e)\cap \overline{W}_{\bar{\g}})\times  
   (X\cap \gt g^*)=X\cap \gt g^*.
$$
Thereby it remains to show that  the intersection    
$X\cap \gt g^*$ has no irreducible components of dimension bigger than
$\dim V_{\gt g,k}-\rk\gt g+\rk\bar{\gt g}$. 

The description of $X$ in terms of tuples 
$\bar s:=(s_1,\ldots,s_k)$ with  $s_i\in \mathbb Z_{\ge 0}$
is given in \cite[Lemma~5.2]{ppy}. 
Denote by
$X_{\bar s}$ the subspace of $V_k^*$ consisting of all 
$\gamma\in V_k^*$ such that $\,\xi_{k-i+1}^{i,d_i-t}(\gamma)=0$ for 
$0\le t<s_i$. The variety $X$ is a union of linear subspaces
$X\,=\,\bigcup_{|\bar s|=d_k+1}\,X_{\bar s}$, where 
$|\bar s|:=s_1{+}s_2{+}\ldots{+}s_k$.  In particular, 
all irreducible components of $X$ have dimension equal to
$\dim V_k-(d_k+1)$. Then restricted to $\gt g^*$ not all of the
linear equations $\,\xi_{k-i+1}^{i,d_i-t}=0$ stay independent, 
$\,\xi_{k-i+1}^{i,d_i-t}$ becomes proportional to 
$\,\xi_{i}^{k-1+1,d_i-t}$, and if $k$ is even, 
then $\xi_{\ell}^{\ell,t}$ with $2\ell=k$ and  even $t$ vanishes 
on $\gt g^*$ completely. Summing  up, 
each component of $X\cap\gt g^*$ has dimension greater than or equal to  
$\dim V_{\gt  g,k}-r$, where $r=(d_k+1)/2$ if $d_k$ is odd, 
$r=d_k/2$ if $d_k$ is even and $k$ is odd, and finally 
if both $d_k$ and $k$ are even, then $r=(d_k+1)/2$. 
In any case, $r=\rk\gt g-\rk\bar{\gt g}$. 
Therefore we can find a subspace $W_{\gt g,k}\subset V_{\gt g,k}^*$ such 
that $X\cap W_{\gt g,k}=0$ and $\dim W_{\gt g,k}=\rk\gt g-\dim\overline{W}_{\bar{\g}}$.
The required subspace $W_{\g}$ is equal to $\overline{W}_{\bar{\g}}\oplus W_{\gt g,k}$.
\end{proof}

Each component of $\cN_F(e)$ is a conical Zariski closed 
subset of $\gt g_e^*$ and we found a subspace $W_{\gt g}\subset\gt g_e^*$
of dimension
$\rk\gt g$ such that $\cN_F(e)\cap W_{\gt g}=\{0\}$. Hence  

\begin{cl}\label{null-cone31}
%%Let $e$ be a nice even  nilpotent element in either orthogonal or
%%symplectic Lie algebra  $\gt g$. 
All irreducible components of $\cN_F(e)$ have
codimension $\rk\gt g$ in $\g_e^*$ and 
${{^e\!}F}_1,\ldots,{{^e\!}F}_{\rk \gt g}$ is a regular sequence in 
$\mathbb \cS(\gt g_e)$.
\end{cl}

Clearly $\cN(e)$ is a subset of $ \cN_F(e)$ and 
each irreducible component of $\cN(e)$ has dimension grater 
or equal than $\dim\gt g_e-\rk\gt g$. Therefore we get the following. 

\begin{cl}\label{null-cone3}
All irreducible components of the null-cone $\cN(e)\subset\gt g_e^*$ have
codimension $\rk\gt g$ in $\g_e^*$.
\end{cl}

Let $X\subset\mathbb A^d_{\mathbb F}$ be a Zariski
closed set and let $x=(x_1,\ldots,x_d)$ be a point of $X$. Let $I$
denote the defining ideal of $X$ in the coordinate algebra
${\mathcal A}=\mathbb F[X_1,\ldots, X_d]$ of 
$\mathbb A^d_{\mathbb F}$. 
Each nonzero $f\in\mathcal A$ can be expressed as a polynomial
in $X_1-x_1,\ldots,X_d-x_d$, say $f=f_k+f_{k+1}+\cdots$, where $f_i$
is a homogeneous polynomial of degree $i$ in 
$X_1-x_1,\ldots, X_d-x_d$ and $f_k\ne 0$. We set ${\rm in}_x(f):=f_k$ 
and denote by
${\rm in}_x(I)$ the linear span of all ${\rm in}_x(f)$ with 
$f\in I\setminus\{0\}$. This is an ideal of $\mathcal A$, and the affine
scheme $TC_x(X):=\,{\rm Spec}\,{\mathcal A}/{\rm in}_x(I)$ is called
the {\it tangent cone} to $X$ at $x$.

If $\gt g$ is of type $D$, then 
$n=2q$ and $F_{p}=P^2$, where $P$ is the Pfaffian.
Set $H_i:={^e\!}F_i$ for all $i$ in types $B$, $C$,
and for $2i<n$ in type $D$; and in type $D$ set 
in addition $H_q:={^e\!}P$.
In exactly the same way as in \cite[Subsection 5.4]{ppy},
one can obtain another  corollary.   

\begin{cl}\label{nilcone}
Let $\cN$ be the nilpotent cone of $\gt g$ and
$F_i$ as above. Suppose that $\gt g$ and a 
nilpotent element $e\in\gt g$ satisfies the assumptions 
of this section. 
Set  $r=\dim\g_e$. Then 
$TC_e({\cN(\gt g)})
\cong\,{\mathbb A}^{\dim\gt g-r}_{\mathbb F}\!\times\,{\rm Spec}\,\cS(\g_e)/({H_1},
  \ldots,{H_{\rk\gt g}})$ 
as affine schemes.
\end{cl}

\begin{qn}
Suppose that $\gt g=\gt{so}(\VV)$ and $i'=i$ for 
a nilpotent element $e\in\gt g$.
Is it true that $H_1,\ldots,H_{\rk\gt g}$ generate 
the whole algebra of symmetric $\gt g_e$-invariants?
The first step is to show that generic fibres of the 
morphism $\gt g_e^*\to{\rm Spec}(\ff[H_1,\ldots,H_{\rk\gt g}])$
are connected. Then the subalgebra 
$\ff[H_1,\ldots,H_{\rk\gt g}]$ will be algebraically closed in
${\mathcal S}(\gt g_e)$, see Theorem~\ref{fibre} in 
the Appendix. Since it has the right transcendence degree, 
$\ind\gt g_e$, it will be shown that at least   
${\mathcal S}(\gt g_e)^{\gt g_e}\subset{\rm Quot}\,\ff[H_1,\ldots,H_{\rk\gt g}]$.
\end{qn}

Related, but a slightly different question, 
is whether ${\mathcal S}(\gt g_e)^{\gt g_e}$ 
is free for the nilpotent elements considered above
(in the orthogonal case, the symplectic case is covered 
by \cite{ppy}). According to  
Kac's generalisation of Popov's conjecture, see footnote~1 on page~192 in 
\cite{some},  it should be. 

\appendix   
\section{When generic fibres of a morphism are connected}
\label{vin-sloi}
\setcounter{equation}{0}

Let $\bbk$ be an algebraically closed 
field of characteristic zero. Suppose that we have 
a dominant morphism $\varphi:\enskip X\to Y$ of 
irreducible affine varieties. Regard 
$\bbk[Y]$ as a subalgebra of $\bbk[X]$ and
$\bbk(Y)$ is a subfield of $\bbk(X)$. 
Let us say that $\bbk[Y]$ is algebraically closed in $\bbk[X]$, if each 
element of $\bbk[X]$, which is algebraic over $\bbk(Y)$, 
lies in $\bkk(Y)$.
The following theorem is probably  very well known. 
The proof given below is due to E.B. Vinberg, 
who explained it to his students at the  Moscow University
some twenty years ago.

\begin{thm}\label{fibre}
Generic fibres of 
$\varphi$ are connected if and only if\  $\ \bbk[Y]$ is algebraically closed in $\bbk[X]$. 
\end{thm}
\begin{proof} 
Suppose first that $\bbk[Y]$ is algebraically closed in $\bbk[X]$.
The algebra $\bbk[Y]$ is finitely generated  by the assumptions on $Y$. 
Let us choose a finite set of generators and let
$\mK\subset\bbk$ be a subfield generated by their coefficients. 
Then  $\varphi$ is defined over $\mK$. 

In this proof we say that a point $y\in Y$ is {\it generic}
if the corresponding map $y:\enskip \mK[Y]\to\bbk$  is 
a monomorphism. 
Informally speaking, 
being generic means that the coordinates of $y$ are very  
transcendental elements of $\bbk$ with respect 
to the subfield $\mK$.  
These generic $y$'s form a dense, not necessary  
open, subset. Since the points 
$u\in Y$ such that $\varphi^{-1}(u)$ is connected 
form a closed subset,  is suffices to 
prove that $\varphi^{-1}(y)$ is connected
for each generic $y$.  

Suppose  $y$ is generic in the above sense. Then 
$$
\bbk[\varphi^{-1}(y)]=\mK[X]\otimes_{\mK[Y]}\bbk=\mK[Y]^{-1}\mK[X]\otimes_{\mK(Y)}\bbk,
$$
where $\mK[Y]$ is embedded into $\bbk$ by $y$ and the last equality holds because 
all elements of $\mK[Y]$ are invertible. 

Note that 
a $\mK(Y)$-algebra $\mK[Y]^{-1}\mK[X]$ 
contains no zero-divisors. (Indeed, if $pq=0$ in $\mK[Y]^{-1}\mK[X]$, 
then multiplying $p$ and $q$ by suitable {\it invertible} elements of 
$\mK[Y]$, we may assume that $p,q\in \mK[X]$. Hence either $p$ or $q$ 
is zero.) This property might not be preserved by the field extension 
$\mK\subset\bbk$. Nevertheless, 
there are no nilpotent elements in $\bbk[\varphi^{-1}(y)]$. In other words, 
a generic fibre is reduced.  If the fibre over $y$ is not connected, then over 
some Galois extension $\mK(Y)\subset L$, the algebra 
$A:=\mK[Y]^{-1}\mK[X]\otimes_{\mK(Y)} L$ decomposes into a direct sum of
indecomposable  ideals 
$$
A=A_1\oplus\ldots \oplus A_m  \qquad \qquad \text{ with } \ m>1.
$$
Let $\Gamma$ be the Galois group of the extension $\mK(Y)\subset L$. 
Then $\mK[Y]^{-1}\mK[X]=A^{\Gamma}$. 
Since this algebra contains no 
zero-divisors, it could not be a dierct sum of 
two non-trivial ideals. 
On the other hand, each $\Gamma$-orbit in the set of ideals
$A_i$ gives rise to an ideal of $A^{\Gamma}$.
Therefore $\Gamma$ acts transitively on the set $\{A_i\mid i=1,\ldots,m\}$. 
Let $\Delta\subset\Gamma$ 
be the normaliser of $A_1$. Note that $|\Gamma/\Delta|=m$, 
hence $\Delta$ is a proper subgroup. 

Choose a subset $\{\gamma_2,\ldots,\gamma_m\}\subset\Gamma$ such that 
$A_i=\gamma_i{\cdot} A_1$. If $a\in A^{\Gamma}$, then 
$a=(a_1,\gamma_2{\cdot}a_1,\ldots,\gamma_m{\cdot}a_1)$, 
where $a_1\in A_1^{\Delta}$. Thus 
%% $A^{\Gamma}$ is contained in a subalegbra 
%%$A_{\rm dg}\cong A_1$ of $A$, which is 
%%%embedded in $A$ diagonally. Moreover 
$\mK[Y]^{-1}\mK[X] \cong A_{1}^{\Delta}$. 
The field $L$ is embedded into $A_1$ and into any of 
the other ideals. Threfore $L^{\Delta}$ is embedded into 
$A_1^{\Delta}$.   
We get a  non-trivial extension of $\mK(Y)$, 
which is contained in $\mK[Y]^{-1}\mK[X]$, i.e.,
$\mK(Y) \subset L^{\Delta} \subset \mK[Y]^{-1}\mK[X]$.
This means that neither $\mK[Y]$ nor $\bbk[Y]$ is algebraically 
closed in $\mK[X]$ or $\bbk[X]$, respectively.   
A contradiction! 

Now suppose that  there exists  $f\in\bbk[X]$, which
is algebraic over $\bbk(Y)$, but is not an element of   
$\bbk(Y)$. Then
there is an open subset $U\subset Y$ such that 
$f$ takes a finite number of values, more than one, 
on each fibre $\varphi^{-1}(y)$ with $y\in U$. 
These values correspond to distinct connected components of 
$\varphi^{-1}(y)$. 
\end{proof}

\begin{rmk}\label{fconv}
Generic fibres of $\varphi$ are irreducible if and only if 
the field $\bbk(Y)$ is algebraically closed in the field $\bbk(X)$,
see e.g. \cite[Chapter~2, Section~6.1]{Shaf}. In case $X$ and 
$Y$ are normal, 
connectedness of generic fibres implies irreducibility, see 
\cite[Proposition~4]{Vanya}. 
In general, this is not true.  
\end{rmk}

Here is an example taken form \cite{Vanya} 
of a dominant morphism with connected 
but reducible generic fibres. 

\begin{ex}Let $X\subset{\mathbb A}^3_{\bbk}$ be the irreducible hypersurface 
defined by the equation $x^2=y^2z$. 
Consider the morphism from $X$ to $Y=\bbk$ given by 
$(x,y,z)\mapsto z$.  
For any $c\ne 0$, the fibre over  $c\in\bbk$  consists of two intersecting lines.
Hence it is connected and reducible. 
The set of intersection points $(0,0,z)$ coincides with the 
singular locus of $X$. Evidently, $X$ is not normal. 
\end{ex}

\begin{comment} 
\begin{ex}  Let $X\subset{\mathbb A}^3_{\bbk}$ be a hypersurface 
defined by $x^3-3xyz+y^3z+z^2=0$. It is irreducible. 
Consider a morphism from $X$ to $Y=\bbk$ given by 
$(x,y,z)\mapsto z$.  A fibre over $c\in\bbk$ is the zero set 
of 
$$
x^3-3cxy+cy^3+c^2=(x+c_1y+c_1^2)(x+c_2y+c_2^2)(x+c_3y+c_3^2),
$$  
where $c_1,c_2,c_3$ are the roots of $w^3=c$. 
If $c\ne 0$, then the fibre over it consists of three lines and  
is reducible.  The first and the second lines intersect at the point 
$(x,y)=(c_1c_2,c_3)$, similar the second and the third --- at $(c_2c_3,c_1)$.
Therefore the fibre is connected. 

!!!! nor normal !!! singularities given by $x=y^2$ !!!!
\end{ex}
\end{comment}


\begin{thebibliography}{3}

\bibitem{Vanya}
{\rusc I.V.\,Arzhancev}, {\rus O de{\u\i}stviyah reduktivnyh grupp
s odnoparametricheskim seme{\u\i}stvom sferi\-cheskih orbit}, 
{\rusi Mat. Sbornik}, {\bf 188}\,(1997), {\rus N0}\,5,
3--20 (Russian). English
translation: {\sc I.V.\,Arzhantsev}, 
On the actions of reductive groups with a one-parameter family of spherical orbits,
{\it Sb. Math.}, {\bf 188}\,(1997), no.\,5, 639--655.
 
\bibitem{bb}
{\sc J.\,\,Brown, J.\,\,Brundan},
Elementary invariants for centralisers of nilpotent matrices, \\
{\tt arXiv:math.RA/0611024}.

\bibitem{cm}
{\sc D.\,Collingwood} and {W.\,McGovern}, ``Nilpotent orbits in
semisimple Lie algebras'', Mathematics Series, Van Nostrand Reinhold, 1993.

\bibitem{cush-rob}
{\sc R.~Cushman} and {\sc  M.~Roberts},
Poisson structures transverse to coadjoint orbits, 
{\it Bull. Sci. Math.}  {\bf 126}(2002), 525--534.

\bibitem{gg}  {\sc W.L.~Gan} and {\sc V.~Ginzburg},
Quantization of Slodowy slices, {\it Intern. Math. Res. Notices},
{\bf 5}\,(2002), 243--255.

\bibitem{gin}
{\sc V.\,Ginzburg}, 
Principal nilpotent pairs in a semisimple Lie algebra. I,
{\it Invent. Math.}, {\bf 140}\,(2000), no.\,3,
511--561.
%%%%%%{\it the one with an appendix by Panyushev-Elashvili.}

\bibitem{graaf}
{\sc W. de Graaf},
Computing with nilpotent orbits in simple Lie algebras of exceptional type, \\
{\tt http://arXiv.org/math.GR/0702193}.

\bibitem{ja}
{\sc J.C.\,Jantzen}, Nilpotent orbits in representation theory, in:
B.\,Orsted (ed.), ``Representation and Lie theory'', Progr. in
Math., {\bf 228}, 1--211, Birkh\"auser, Boston 2004.

\bibitem{some}
{\sc V.G.\,Kac}, Some remarks on nilpotent orbits, 
{\it J. Algebra}, {\bf 64}\,(1980), no.\,1, 190--213. 

\bibitem{kost63}
{\sc B.\,Kostant}, 
Lie groups representations on polynomial rings,
{\it Amer. J. Math.},  {\bf 85}\,(1963), 327--404. 

\bibitem{kur}
{\sc J.F.\,Kurtzke}, Centralizers of irregular elements in reductive algebraic groups,
{\it Pacific J. Math.}, {\bf 104}\,(1983), no.\,1, 133--154.

\bibitem{mus}
{\sc M.\,Must{\u{a}}{\c{t}}a}, 
Jet schemes of locally complete intersection canonical
singularities 
(with an appendix by David Eisenbud and Edward Frenkel),
{\it Invent. Math.},
{\bf 145}\,(2001), no.\,3, 397--424.

\bibitem{NS}
{\sc M.G.~Neubauer} and {\sc B.A.~Sethuraman}, Commuting
Pairs in the Centralizers of 2-Regular Matrices, {\it J.
Algebra}, {\bf 214}\,(1999), no. 1, 174--181.

\bibitem{oomsb}
{\sc A.I.~Ooms} and {\sc M. Van den Bergh},
A degree inequality for Lie algebras with a regular Poisson semi-center,
{\tt arXiv:0805.1342v1\,[math.RT].}

\bibitem{Dima}
{\sc D.~Panyushev}, On the coadjoint representation of 
$\mathbb Z_2$-contractions of reductive Lie algebras, 
{\it Adv. Math.}, {\bf 213}\,(2007), no.\,1, 380--404. 

\bibitem{ppy}
{\sc D.~Panyushev, A.~Premet,}  and {\sc O.~Yakimova}, 
On symmetric invariants of centralisers
in reductive Lie algebras, {\it J. Algebra},
{\bf 313}\,(2007), 343--391.

\bibitem{codim3}
{\sc D.~Panyushev} and  {O.~Yakimova},
The argument shift method and maximal commutative subalgebras
of Poisson algebras, {\it Math. Res. Letters}, 
{\bf 15}\,(2008), no.\,2, 239--249.

\bibitem{Rich-com}
{\sc R.\,Richardson}, Commuting varieties of semisimple Lie algebras
and algebraic groups, {\it Compositio Math.}, {\bf 38} (1979), 311--327.

\bibitem{sek}
{\sc J.~Sekiguchi}, 
A counterexample to a problem on commuting matrices.
{\it Proc. Japan Acad. Ser. A Math. Sci.}, 
{\bf 59}\,(1983), no.\,9, 425--426.

\bibitem{Shaf}
{\sc I.R.\,Shafarevich}, 
Basic algebraic geometry. 1.
Varieties in projective space. Second edition. Translated from the 1988 Russian edition and with notes by Miles Reid. Springer-Verlag, Berlin, 1994.

\bibitem{stein}
{\sc R.\,Steinberg}, 
Conjugacy classes in Algebraic groups, 
Notes by V.\,V.~Deodhar. Lecture Notes in Math., 
{\bf 366}, Springer-Verlag, Berlin-New York, 1974.

\bibitem{int}
{\sc E.B.~Vinberg} and {\sc O.S.~Yakimova}, 
Complete families of commuting functions for coisotropic Hamiltonian actions, 
{\tt  arXiv:math.SG/0511498.}

\bibitem{alan}
{\sc A.\,Weinstein},
The local structure of Poisson manifolds,
{\it J. Differential Geom.}, {\bf 18}\,(1983), no.\,3, 523--557. 

\bibitem{fan}
{\rusc O.~Yakimova}, {\rus Indeks centralizatorov {e1}lementov v
klassicheskikh algebrakh Li}, {\rusi Funkc. analiz i ego prilozh.},
{\bf 40}, {\rus N0}\,1 (2006), 52--64 (Russian). English
translation: {\sc O.~Yakimova}, The centralisers of nilpotent
elements in classical Lie algebras, {\it Funct. Anal. Appl.}, 
{\bf 40} (2006), 42--51.

\end{thebibliography}
\end{document}